\documentclass[11pt]{article}
\usepackage{amsmath,amssymb,amsthm,color,graphicx}
\usepackage{accents}
\usepackage[unicode,breaklinks=true,colorlinks=true]{hyperref}
\usepackage[top=0.9in, bottom=0.8in, left=1in, right=1in]{geometry}
\renewcommand{\baselinestretch}{1.05}

\numberwithin{equation}{section}
\newtheorem{theorem}{Theorem}[section]
\newtheorem*{theorem*}{Theorem}
\newtheorem{corollary}[theorem]{Corollary}
\newtheorem{lemma}[theorem]{Lemma}

\newtheorem{definition}[theorem]{Definition}

\theoremstyle{remark}
\newtheorem{remark}[theorem]{Remark}
\newtheorem*{remark*}{Remark}

\newcommand{\bke}[1]{\left ( #1 \right )}

\newcommand{\bket}[1]{\left \{ #1 \right \}}
\newcommand{\norm}[1]{ \| #1  \|}

\newcommand{\abs}[1]{\left | #1 \right |}

\newcommand\al{\alpha}

\newcommand\de{\delta}

\newcommand\e {\epsilon}

\newcommand\la{\lambda}

\newcommand{\R}{\mathbb{R}}
\newcommand{\RR}{\mathbb{R}}

\newcommand{\Z}{\mathbb{Z}}

\renewcommand{\div}{\mathop{\rm div}}

\newcommand{\esssup} {\mathop{\rm ess\,sup}}
\newcommand{\CKN}{\text{\tiny CKN}}

\newcommand{\td}{\tilde}

\newcommand{\EQ}[1]{\begin{equation}\begin{split} #1 \end{split}\end{equation}}

\newcommand{\loc}{\mathrm{loc}}
\newcommand{\uloc}{\mathrm{uloc}}

\newcommand*{\dt}[1] {\accentset{\mbox{\large\bfseries .}}{#1}}
\renewcommand{\dot}{\dt}

\let\OLDthebibliography\thebibliography
\renewcommand\thebibliography[1]{
  \OLDthebibliography{#1}
  \setlength{\parskip}{1pt}
  \setlength{\itemsep}{2pt plus 0.3ex}
}

\begin{document}

\title{An $\epsilon$-regularity criterion and estimates of 
the regular set \\ for Navier-Stokes flows in terms of initial data}
\author{Kyungkeun Kang%
\thanks{Department of Mathematics, Yonsei University, Seoul 120-749,
South Korea. Email: \texttt{kkang@yonsei.ac.kr}}
\and
Hideyuki Miura%
\thanks{Department of Mathematical and Computing Science, Tokyo Institute of Technology, Tokyo 152-8551, Japan. Email: \texttt{miura@is.titech.ac.jp}}
\and
Tai-Peng Tsai%
\thanks{Department of Mathematics, University of British Columbia,
Vancouver, BC V6T 1Z2, Canada. Email:  \texttt{ttsai@math.ubc.ca}
}}
\date{}%
\maketitle
\begin{abstract}
We prove an $\epsilon$-regularity criterion for the 3D Navier-Stokes 
equations in terms of initial data. 
It shows that if a  scaled local $L^2$ norm of initial data is sufficiently small around the origin, a suitable weak solution is regular in a set enclosed by a paraboloid 
started from the origin. The result is applied to the estimate of 
the regular set for local energy solutions with initial data in 
weighted $L^2$ spaces. We also apply this result to studying  
energy concentration near a possible blow-up time and 
regularity of forward discretely self-similar solutions.
 
\end{abstract}

\section{Introduction}
\subsection{Regular set for suitable weak solutions}

We consider the regularity of weak solutions for the incompressible Navier-Stokes equations
\begin{align}\label{NS}\tag{\textsc{ns}}
\partial_tv-\Delta v +v\cdot\nabla v+\nabla p = 0,
\quad
\div v =0
\end{align}
associated with the initial value $v|_{t=0}=v_0$ with $\div v_0=0$.
 The global in time existence of weak solutions for finite energy
initial data data was proved by Leray \cite{Leray} and Hopf \cite{Hopf}.
Despite a lot of effort since their foundational work,
 the global regularity of the weak solutions remains a
 longstanding open problem.
After the pioneering work by Scheffer \cite{Scheffer77, Scheffer80},
 Caffarelli, Kohn, and Nirenberg \cite{CKN} established  local regularity theory for 
 \textit{suitable weak solutions} which are weak solutions satisfying
\textit{the local energy inequaltity}; see Section 2 for details.
As an application of their celebrated $\epsilon$-regularity criterion,
they showed the following result on the \textit{regular set}:

\medskip\noindent
{\bf Theorem~D\,\cite{CKN}}.
\textit{There exists $\epsilon_0>0$ such that if $v_0 \in L^2(\R^3)$
satisfies
\begin{equation}
\||x|^{-\frac12} v_0\|_{L^2}^2 =\epsilon <\epsilon_0
\label{0405},
\end{equation}
then there exists a suitable weak solution which is regular in the set
$\Pi_{\epsilon_0-\epsilon}$,
where
\begin{align*}
\Pi_{\delta}:=\left\{(x,t): t>\frac{|x|^2}{\delta}   \right\} \qquad \textit{for} \ \delta>0.
\end{align*}
}
\\
This theorem shows that smallness of initial data 
implies regularity of
the solution above a paraboloid with vertex at the origin.
There are at least two interesting features in this result:
No regularity condition (better than $L^2$) is assumed away from the origin and the regularity around the origin is propagated globally in time.
We also note that if the size of $v_0$ tends to 0, the regular set increases and invades a limit set $\Pi_{\epsilon_0}$.
This observation leads to the following questions:
\begin{enumerate}
\item[(a)]
 Can  the size of regular set $\Pi_{\delta}$ be enlarged?
\item[(b)]
 Can the condition \eqref{0405} of initial data be relaxed in terms of 
 regularity and smallness?
\end{enumerate}
The question (a) is addressed by D'Ancona and Luca \cite{DL},
where it is shown that there exists $\delta_0>0$ such that
if $v_0 \in L^2(\R^3)$ satisfies
\begin{equation}
\label{0406}
\||x|^{-\frac12}v_0\|_{L^{2}}< \delta_0 e^{-4L^2}
\end{equation}
for some $L>1$,  \eqref{NS} has a suitable weak solution
which is regular in the set
$$
\Pi_{L\delta_0}=\left\{(x,t): t>\frac{|x|^2}{L\delta_0}   \right\}.
$$
In particular,  the regular set  $\Pi_{L\delta_0}$ invades the whole half space
$\R^3 \times (0,\infty)$ when $v_0$ tends to zero,
though \eqref{0406} still assumes smallness of the data.
One of the goals of this paper is trying to answer questions (a) and (b)
by employing approach based on a framework of \textit{scaled local energy}
explained below.

\subsection{Main result on local-in-space regularity near initial time}
It is known from works \cite{FJR, Weissler, Kato84, GM} that for $v_0 \in L^q(\R^3)$ with $q \ge 3$, \eqref{NS} has a (unique) mild solution defined on some short time interval.
Motivated by the problem for constructing large forward self-similar solutions to \eqref{NS},
Jia and \v Sver\'ak \cite{JS} asked which condition this result can be localized in space. 
Let $B_r(x)=\{y \in \R^3:\, |x-y|<r\}$ and $B_r=B_r(0)$. Then
 their question can be stated as follows:
\begin{enumerate}
\item[(c)]
If $v_0$ is a general initial data for
which suitable weak solutions $v$ is defined and $v_0|_{B_2} \in L^q(B_2)$,
can we conclude that $v$ is regular in $B_1 \times [0, t_1)$ for some time $t_1 > 0$?
\end{enumerate}
Although non-local effect of the pressure might prevent the solution 
from having the same amount of the regularity
as the one for the heat equation, such effect is expected to be handled 
at least for a short time and $q \ge 3$. 
Indeed this question is settled affirmatively for the subcritical case $q>3$ in \cite{JS} and for the critical case $q=3$ 
in \cite{BP,KMT20}
\footnote{See also \cite{Tao} for the condition on the initial enstrophy and
\cite{BP} for further extension to the $L^{3,\infty}$ space and the critical Besov spaces.}.
Notice that the results for the critical case have some similarities
with Theorem D in \cite{CKN}.  Namely, the assumptions for the initial data
 ensure critical regularity at the origin and 
they lead to local-in-space regularity.
As the first main result of this paper, we present a new type of local-in-space regularity estimate, which guarantees regularity in the set
like $\Pi_{\delta}$.
In order to formulate it,  define \textit{the scaled local energy}
of the initial data by
$$
 N_0=N_0(v_0):= \sup_{r\in (0,1]} \frac{1}{r}\int_{B_r} \abs{v_0(x)}^2dx,
$$
which plays a central role in this paper.
\begin{theorem}
\label{thm:local}
Let $(v,p)$ be a suitable weak solution in $B_2\times (0, 4)$
 with the initial data $v_0 \in L^2(B_2)$ in the sense that
 $\lim_{t\rightarrow 0+}\|v(t)-v_0 \|_{L^2(B_2)}=0$.
 Assume that
\begin{equation}
\label{0509}
M:=\norm{v}^2_{L^{\infty}_t(0,4;L^2_x(B_2))} + \norm{\nabla v}^2_{L^2(B_2\times (0,4))}
+\norm{p}_{L^{\frac32}(B_2\times (0,4))}<\infty
\end{equation}
and  that
\begin{equation}
\label{0407}
 N_0 \le \epsilon_*.
\end{equation}
Then there exists $T=T(M) \ge \frac{c}{1+M^{18}}$ such that $v$ is regular in the set
\[
\Gamma=\bket{(x,t) \in B_1\times (0,1)\,: cN_0^2\abs{x}^2\le t<T}
\]
and satisfies
\[
\abs{v(x,t)}\le \frac{C}{t^{\frac{1}{2}}} \qquad  \mbox{for} \ (x,t) \in 
\Gamma,
\]
where $\epsilon_*$, $c$, and $C$ are positive absolute constants. 
\end{theorem}
\begin{remark}(1)\,This theorem shows that smallness of the scaled energy implies regularity above a paraboloid for a short time.
It may be  viewed as an $\epsilon$-regularity criterion \textit{in terms of the initial data}. 
\\
(2)\,One can relate Theorem \ref{thm:local} to results in \cite{CKN, BP,KMT20}
by noting that
$$
N_0
\le
C \min \{ \|v_0\|^2_{L^3(B_1)}, \|v_0\|^2_{L^{2,-1}(B_1)}\},
$$
where
\begin{equation*}
\|v_0\|_{L^{2,\alpha}(\Omega)}:=\||x|^{\frac{\alpha}{2}}v_0\|_{L^2(\Omega)}
\end{equation*}
for $\alpha\in \R$ and $\Omega \subset \R^3$.
Thus \eqref{0407} holds if either $L^3$ norm or $L^{2,-1}$ norm is small in $B_1$.
Hence our theorem can be regarded as a local version of Theorem D
in \cite{CKN}.
\\
(3)\,By rescaling and translation, it is easy to see that 
if the data satisfies $\sup_{r\in (0,R]} \frac{1}{r}\int_{B_r(x_0)} |v_0|^2 \le \epsilon_*$ for some $x_0 \in B_2$ and $R>0$, then $v$ is regular in the set enclosed by a paraboloid 
with vertex at $(x_0,0)$ and the plane $t=TR^2$.
In particular, if  $\|v_0\|_{m^{2,1}(B_2)} \le \epsilon_*$,
there exists $T=T(M)>0$ such that the suitable weak solution is regular in $B_1 \times (0,T)$
and satisfies $|v(x,t)| \le \frac{C}{\sqrt{t}}$ in $B_1 \times (0,T)$.
Here $\|\cdot\|_{m^{2,1}}$ denotes the local Morrey norm:
$$
\|f\|_{m^{2,1}(\Omega)} :=\sup_{x_0 \in \Omega, \,r\in (0,1]}
\left( \frac{1}{r}\int_{B_r(x_0) \cap \Omega} \abs{f(x)}^2dx \right)^{\frac12}.
$$
By well-known relations $m^{2,1} \supset L^{3,\infty} \supset L^3$, we see that this extends
results of \cite{KMT20} ($L^3$ case) as well as \cite{BP} ($L^{3,\infty}$ case).
Recently it is shown in \cite[Theorem 1.2]{BT19} that 
the local energy solution is regular in time interval $(0,cR^2]$,
 provided
$\displaystyle{\sup_{x_0 \in \R^3, \,r\in (0,R]}}$$\frac{1}{r}\int_{B_r(x_0)} \abs{v_0}^2$ is small. 
It should be pointed out that the Morrey space 
is introduced in the context of the Navier-Stokes equations 
by Giga and Miyakawa \cite{GM89} 
for the study of vortex filaments in $\R^3$. Well-posedness of \eqref{NS} in the Morrey space $m^{2,1}$
(and its homogeneous version $M^{2,1}$) 
is studied by Kato \cite{Kato92} in $\R^n$ 
and Taylor \cite{Taylor} in compact Riemannian manifolds.
Recently, solutions from vortex filaments with arbitrary circulation
are constructed in
 \cite{BGHG}; see references therein for  development of this direction.

\end{remark}

\medskip

The proof of Theorem \ref{thm:local} is based on a local-in-space a priori estimate for the scaled energy of $(v,p)$ defined by
\begin{equation}
\label{scaled energy}
E_r(t):=\sup_{0<s<t}\frac{1}{r}\int_{B_r} \abs{v(s)}^2+\frac{1}{r}\int_0^t\int_{B_r} \abs{\nabla v}^2
+\frac{1}{r^2} \int_0^t\int_{B_{r}}\abs{p}^{\frac{3}{2}}.
\end{equation}
 Our strategy is partially inspired
by a uniformly local $L^2$ estimate established in 
the fundamental work
\cite{LR} of Lemari\'e-Rieusset.
However, in contrast to his estimate, our
a priori estimate guarantees scale-critical regularity at the origin
 so that the $\epsilon$-regularity criterion 
of \cite{CKN} can apply.
Furthermore, since we make assumptions only locally in space,
nontrivial modification is required to deal with the 
non-local effect of the pressure.
It should be noted that our strategy is also different from
 works \cite{JS, BP,KMT20}
and ours provides rather simple and transparent approach.
Indeed, in these previous works, the proof relied on decomposition of
the solution in $B_2$ into two parts;
one is the (regular) solution to \eqref{NS} for the data $\mathcal{B}(v_0|_{B_1}) \in L^q$ where
$\mathcal{B}$ is the
Bogovskii extension operator to $\R^3$, and the other is the solution to the perturbed Navier-Stokes equation for the data $v_0-B(v_0|_{B_1})$.
The main task in previous works is to show the regularity of the perturbed part,
and hence they developed local regularity theory for the perturbed Navier-Stokes equations.
In contrast, we estimate the solution of \eqref{NS} directly
without using such decomposition.

\subsection{Applications}
\subsubsection{Regular set}
We now return to the questions (a) and (b) concerning the regular set.
In order to state our result, it is natural and convenient to use
the notion of \textit{local energy solutions}
 introduced by Lemari\'e-Rieusset \cite{LR} and later slightly modified in \cite{KS, JS, BT19}.
The local energy solution is a suitable weak solution of \eqref{NS} defined in $\R^3$
which satisfies certain uniformly local energy bound and pressure representation;
see Definition \ref{def:localenergy} for the details.
In this context, let us recall the uniformly local $L^q$ spaces for $1\le q <\infty$.
We say $f\in L^q_{\uloc}$ if $f\in L^q_\loc(\R^3)$ and
\EQ{\label{Lq-uloc}
\norm{f}_{L^q_{\uloc}} =\sup_{x \in\R^3} \norm{f}_{L^q(B_1(x))}<\infty.
}
\medskip
Local in time existence of local energy solutions for initial data in $L^2_{\uloc}$
and also global existence for initial data in $E^2:=\overline{C_0^\infty}^{L^2_{\uloc}}$
are established in \cite{LR}. 
One of the advantage of the local energy solution
is it can be defined even for infinite energy data; see  \cite{KT, LR2, JST} and references therein
for further developments and its applications,
and \cite{MMP} for  the local energy solutions in the half space.
We also define the global version of the scaled energy by
\begin{equation*}
\dot{N}_0:=\sup_{r>0}\frac{1}{r}\int_{B_r} \abs{v_0(x)}^2dx
\end{equation*}
Here note that $\dot{N}_0$ is invariant under
the Navier-Stokes scaling:\,$u(x,t) \mapsto  u_\lambda(x,t):=\lambda u(\lambda x,\lambda^2 t)$.
The following result shows the estimates of the regular set for
the local energy solution for initial data
with small scaled energy and also about that for
 large data in $L^{2,-1}(\R^3)$:
\begin{theorem}
\label{thm:0430}
Let $(v,p)$ be a local energy solution in $\R^3 \times (0,\infty)$
for  the initial data $v_0 \in L^2_{\uloc}(\R^3)$.
\\
{\rm (i)}\,There exist  absolute constants $\epsilon_*$ and $c$ such that
if  $v_0$ satisfies
\begin{equation}
\label{decay'}
\sup_{x_0\in \R^3}\sup_{r\ge 1}\frac{1}{r}\int_{B_r(x_0)} 
\abs{v_0(x)}^2dx <\infty,
\end{equation}
and if
\begin{equation}
\label{03112'}
\dot{N}_0\le \epsilon_*,
\end{equation}
then $v$ is regular in the set
$$
\bket{(x,t) \in \R^3\times (0, \infty)\,:\,
c\dot{N}_0^2 \abs{x}^2 \le t}.
$$
\\
{\rm (ii)}\,For any $v_0 \in L^{2,-1}(\R^3)$ there exist
positive constants $T(v_0)$ and $c(v_0)$ such that $v$ is regular in the set
$$
\bket{(x,t) \in \R^3 \times (0,\infty)\,:\,c(v_0)|x|^2\le t<T(v_0)}.
$$
\end{theorem}
Note that $\dot{N}_0 \le  \|v_0\|_{L^{2,-1}(\R^3)}^2$ holds
and that the condition \eqref{decay'} only assumes some mild decay of the data at infinity, and it is satisfied for the data in $L^2$ 
or even for  infinite energy data in the homogeneous 
Morrey space $M^{2,1}$ (see \eqref{04202}).
Therefore (i) relaxes the assumptions for the initial data
given in \cite{CKN, DL}. It also refines the convergence rate of the regular set to $\R^3 \times (0,\infty)$ in \cite{DL} as $\|v_0\|_{L^{2,-1}}$ tends to zero. 
Thus Theorem \ref{thm:0430} extends \cite[Theorem D]{CKN} and \cite{DL}. Moreover, \cite[Theorem D]{CKN} and \cite{DL} are existence results for initial data satisfying the conditions, while Theorem \ref{thm:0430} is a regularity result for \emph{any} solution for such data.

Global existence of local energy solutions for data 
satisfying \eqref{decay'} is proved in \cite{BT19}. 
See Theorem \ref{0311} for further results including \emph{eventual regularity} for $v_0 \in L^{2,-1}(\R^3)$ also satisfying
\eqref{decay'}.

\medskip

The following corollary concerns estimates of regular set 
for the data in the
weighted space $L^{2,\alpha}$ with $\alpha> -1$, which generalize  
the classical result \cite{Leray} of Leray and \cite[Theorem C]{CKN} 
for the cases $\alpha=0$ and $\alpha=1$, respectively. 
\begin{corollary}
\label{thm:weighted}
Let $(v, p)$ be a local energy solution for the initial data
in $L^2_{\uloc}(\R^3)$.
\\
{\rm (i)}~Assume that $v_0 \in L^{2,\alpha}(\R^3)$ for some $\alpha \ge 0$.
Then $v_0\in L^2(\R^3)$ and there is $K=K(\|v_0\|_{L^2}, \|v_0\|_{L^{2,\alpha}})$ such that  $v$ is regular in the set
$$
\bket{(x,t) \in \R^3 \times (0,\infty)\,:\, t\ge \min \{ K |x|^{-2\alpha}, C_0\|v_0\|_{L^2}^4 \} }.
$$
\\
{\rm(ii)}~Assume that $v_0  \in L^{2,\alpha}(\R^3)$ for some $\alpha \in (-1,0)$ and that \eqref{decay'} holds.
Then there exist $K=K(\|v_0\|_{L^{2,\alpha}})$ and $T=T(\|v_0\|_{L^{2,\alpha}})$
such that 
$v$ is regular in the set
$$
\bket{(x,t) \in \R^3 \times (0,\infty)\,:\, t\ge \max \{ K |x|^{-2\alpha}, T \} }.
$$
\end{corollary}

\begin{remark}
(1)~Regularity near the initial time is not guaranteed in this case.
Note that the condition $v_0 \in L^{2,\alpha}$ for $\alpha> -1$ assumes 
faster decay at infinity than the case for $\alpha=-1$, 
but less regularity at the origin.
\\
(2)~Global existence of local energy weak solutions for data in $L^2_{\uloc} \cap L^{2,\beta}$ for $\beta \in (-2,0)$ is proved by %
\cite{FL}. 
See also more recent work \cite{BKT} for further generalization.
\end{remark}

\subsubsection{Energy concentration near a possible blow up time}
Regularity theory developed in this paper has applications to other problems as well. 
The following theorem shows concentration of the scaled local energy
in a shrinking ball near a possible blow-up time.
\begin{theorem}
\label{concentration}
Let $v$ be a local energy solution and $T_* \in (0,\infty)$ the maximal 
time so that $v \in C((0,T_*);L^\infty(\R^3))$. 
There exists $\epsilon_*>0$ such that the following holds:
\\
{\rm (i)}~There exist $S_0>0$ and $x(t) \in \R^3$ for $t \in (0,T_*)$
such that
\begin{equation}
\frac{1}{\sqrt{T_* -t}}\int_{B_{\sqrt{\frac{T_*-t}{S_0}}}(x(t))}\abs{v(x,t)}^2dx>\epsilon_*.
\label{04082}
\end{equation}
\\
{\rm(ii)}\,Suppose that $(x,t)=(0, T_*)$ is a singular point
and  that 
the type I condition
\begin{equation}
\label{0318}
\sup_{x_0 \in \R^3}\sup_{r \in (0,r_0]}\sup_{t\in (T_*-r^2, T_*)} \frac{1}{r}\int_{B_r(x_0)}\abs{v(x,t)}^2dx=:M_*<\infty
\end{equation}
holds for some $r_0>0$. Then, there exist $S=S(M_*)>0$ and $\delta_*>0$ such that
\begin{equation}
\label{0401}
\frac{1}{\sqrt{T_* -t}}\int_{B_{\sqrt{\frac{T_*-t}{S}}}}\abs{v(x,t)}^2dx>\epsilon_* \qquad \mbox{for}  \ t\in (T_*-\delta_*,T_*).
\end{equation}
\end{theorem}

\begin{remark}(1)\, Part (i) of Theorem \ref{concentration} is a restatement of \cite[Theorem 8.2]{BT20} with slightly different conditions on $v$.
Our new proof is based on Theorem \ref{large data}.
\\
(2)\, There is a lot of literature on the behavior 
of the critical norm near the blow up time. However, 
to our knowledge, 
the behavior of the critical norm in shrinking balls has not 
been studied until recently. 
Li, Ozawa, and Wang \cite{LOW} showed concentration of 
$L^3$ norm at the blow up time in a ball $B(c\sqrt{T_*-t})$
along some sequence., i.e, they showed that there exist $t_n \uparrow T_*$ and $x_n \in \R^3$
such that 
$$
\|v(t_n)\|_{L^3(B_{c(T_*-t_n)}(x_n))} > \gamma_0
$$
holds with some absolute constants $c$, $\gamma_0>0$. 
In \cite[Corollary 1.1]{MMP2}, $L^3$ concentration in $B_{c(T_*-t)}(x(t))$
for \textit{all} $t <T_*$ is shown with some trajectory $\{x(t)\}_{t<T_*}$.
Statement (i) improves \cite[Corollary 1.1]{MMP2} since
the scaled energy is bounded by $L^3$ norm in $B_{\sqrt{\frac{T_*-t}{S_0}}}$.
We note that similar concentration phenomena are studied extensively
for nonlinear dispersive equations with $L^2$-critical nonlinearities; 
see \cite{MT, B, KPV}. 
\\
(3)\,This version of \textit{type I condition} \eqref{0318} is introduced in \cite{BP}.
It is implied from classical \textit{type I conditions}
such as $$
|v(x,t)| \le C(T_*-t)^{-\frac12}, \qquad \quad \ \quad
|v(x,t)| \le C|x|^{-1},$$
or the uniform Morrey bound $\sup_{t\in (0,T_*)}\|v(t)\|_{m^{2,1}}\le C < \infty$; see \cite{ZS, BP} for details.
In \cite{BP}, a concentration
of the $L^3$ norm at the singular point; i.e.,
$\|v(t)\|_{L^3(B_{\sqrt{\frac{T_*-t}{S'}}})} > \gamma_*$ with some
$\gamma_*>0$, $S'=S'(M_*)>0$ 
is shown under the condition \eqref{0318}.
From \eqref{0401} and the H\"{o}lder inequality, we can
also deduce the similar concentration estimate:
$$
\int_{B_{\sqrt{\frac{T_*-t}{S}}}}\abs{v(x,t)}^3 dx \ge
\sqrt{\frac{S}{T_* -t}}\int_{B_{\sqrt{\frac{T_*-t}{S}}}}\abs{v(x,t)}^2 dx
>\sqrt{S}\epsilon_*.
$$
\end{remark}
\subsubsection{Regularity of discretely self-similar solutions}
As another application of our result, we study regularity 
of discretely self-similar solutions. 
A solution $v$ defined in $\R^3 \times (0,\infty)$
is called (forward) self-similar (SS) if $v^\lambda(x,t)=v(x,t)$ for all $\lambda>0$ and is discretely self-similar with factor $\lambda$ (i.e.~$v$ is $\lambda$-DSS) if this scaling invariance holds for a given $\lambda>1$. Similarly, $v_0$ is self-similar (or~$(-1)$-homogeneous) if $v_0(x)=\lambda v_0(\lambda x)$ for all $\lambda>0$ or $\lambda$-DSS if this holds for a given $\lambda>1$.
 In \cite{JS},  Jia and \v Sver\'ak constructed self-similar solutions for H\"older continuous (away from the origin) initial data.
This result has been generalized in a number of directions  \cite{BT17,BT2,BT3,BT5, Chae-Wolf,KT-SSHS,LR2,Tsai}; see also the survey \cite{JST}. Among other results, Chae and Wolf  
\cite{Chae-Wolf} constructed  DSS solutions for any DSS data 
in $L^2_{\loc}$. Another construction was given by 
Bradshaw and Tsai \cite{BT5}, where the constructed solutions
satisfy the local energy inequality. 
The following result shows regularity 
of DSS solutions for data in $E^2=\overline{C^{\infty}_0}^{L^2_{\uloc}}$ provided $\lambda$ is close to 1.

\begin{theorem}
\label{thm:DSS}
{\rm (i)}~Let $\la>1$ and $v$ be a $\la$-DSS local energy solution of the Navier-Stokes equations \eqref{NS}
with $\la$-DSS data $v_0 \in E^2$.
Then 
\begin{equation}
\label{04202}
\|v_0\|_{M^{2,1}}:=\sup_{x_0 \in \R^3,\,r>0} \Big(
\frac{1}{r} \int_{B_r(x_0)} |v_0(x)|^2dx \Big)^{\frac12} <\infty
\end{equation}
holds and there exists $\mu \in (0,1)$ such that
\begin{equation}
\label{0419}
\sup_{x_0 \in \R^3,\,r\in(0,\mu|x_0|] }
\frac{1}{r}\int_{B_r(x_0)}|v_0(x)|^2 dx
\le
\epsilon_*.
\end{equation}
 Moreover $v$ is regular  in the set
$$
\bket{(x,t) \in \R^3\times (0, \infty)\,:\,
0  < t \le  T_2(\|v_0\|_{M^{2,1}})\mu^2\abs{x}^2}.
$$
\\
{\rm (ii)}~For any $\mu \in (0,1)$, there exists
$\lambda_*=\lambda_*(\mu)\in (1,2)$ such that if any
$\lambda$-DSS  data $v_0\in E^2$
with factor $\lambda \in  (1,\lambda_*]$ satisfies \eqref{0419}, then the $\lambda$-DSS local energy solution $v$ is regular in $\R^3\times (0,\infty)$ and 
\[
\abs{v(x,t)}\le\frac{C}{\sqrt{t}}\quad   {\rm in} \ \R^3\times (0,\infty)
\]
holds for a positive constant $C=C(v_0)$.
\end{theorem}
\begin{remark} 
The statement (ii) extends \cite[Corollary 1.5 (ii)]{KMT20} on regularity of DSS solutions for data in $L^{3,\infty}$.
We note that \eqref{04202} was already proved in \cite[Lemma 6.1]{BT19}, where it was shown that for $\lambda$-DSS data,
\eqref{04202} holds if and only if $v_0 \in L^2_{\uloc}(\R^3)$.
On the other hand, \eqref{0419} does not hold for general $v_0 \in L^2_{\uloc}$. See Remark \ref{example} for details.
\end{remark}

\subsection{Outline of the paper and notation}

In section \ref{preliminaries}, we introduce the notion of local energy weak solutions and 
recall the local regularity criterion due to \cite{CKN} 
as well as some technical lemmas.
Section \ref{main results} is devoted to stating and proving our main results including Theorem \ref{thm:local}.
In section \ref{applications}, as applications, local regularity, regular set and a blow-up criterion
are analyzed in terms of scaled local $L^2$ energy. Regularity of discretely self-similar solutions is also studied for locally $L^2$ data.

Throughout this paper, $C \in (0,\infty)$ 
denotes an absolute constant which may change line by line.

\section{Preliminaries}
\label{preliminaries}
In this section, we recall some notions about the weak solution to
\eqref{NS} and some results such as the $\epsilon$-regularity theorems
and a priori estimates for the solutions.

For any domain $\Omega \subset \R^3$ and open interval
$I \subset (0,\infty)$, we say $(v,p)$ is a
suitable weak solution in $\Omega \times I$
if
it satisfies \eqref{NS} in the sense of distributions in
$\Omega \times I$,
\begin{align*}
v \in L^\infty(I; L^2(\Omega)) \cap L^2(I; {\dot H}^1(\Omega)), \quad  p \in L^{3/2}(\Omega \times I),
\end{align*}
and the local energy inequality:
\EQ{\label{CKN-LEI}
&\int_{\Omega} |v(t)|^2\phi(t) \,dx +2\int_0^t\!\! \int_{\Omega} |\nabla v|^2\phi\,dx\,dt
\\&\leq %
\int_0^t\!\!\int_{\Omega} |v|^2(\partial_t \phi + \Delta\phi )\,dx\,dt +\int_0^t\!\!\int_{\Omega} (|v|^2+2p)(v\cdot \nabla\phi)\,dx\,dt
}
for  all non-negative $\phi\in C_c^\infty ( \Omega \times I)$.
Note that no boundary condition is required.

We next define the notion of local energy solutions,
The following definition is formulated in \cite{BT19}, which is slightly revised
from the notions of the local Leray solution defined in \cite{LR},
the local energy solution in \cite{KS} and the Leray solution in \cite{JS}.
We refer to \cite[Section 2]{KMT20} for discussion of their relation.

\begin{definition}[Local energy solutions \cite{BT19}]\label{def:localenergy}
A vector field $v\in L^2_{\loc}(\R^3\times [0,T))$ is a local energy solution to \eqref{NS} with divergence free initial data $v_0\in L^2_{\uloc}$ if
\begin{enumerate}
\item for some $p\in L^{3/2}_{\loc}(\R^3\times [0,T))$, the pair $(v,p)$ is a distributional solution to \eqref{NS},

\item for any $R>0$,

\begin{equation}\label{uniform-energy}
\esssup_{0\leq t<\min \{R^2,T \}}\,\sup_{x_0\in \R^3}\, \int_{B_R(x_0 )} |v|^2\,dx
+ \sup_{x_0\in \R^3}\int_0^{\min\{ R^2, T\}}\!\!\!\int_{B_R(x_0)} |\nabla v|^2\,dx \,dt<\infty,
\end{equation}

\item for all compact subsets $K$ of $\R^3$ we have $v(t)\to v_0$ in $L^2(K)$ as $t\to 0^+$,

\item $(v,p)$ satisfies the local energy inequality
\eqref{CKN-LEI}
for  all non-negative functions $\phi\in C_c^\infty (Q)$
with all cylinder $Q$ compactly supported in $\R^3 \times (0,T)$,
\item for every $x_0\in \R^3$ and $r>0$, there exists $c_{x_0,r}\in L^{3/2}(0,T)$ such that
\EQ{\label{decomposition}
        p (x,t)-c_{x_0,r}(t)&=\frac 1 3 |v(x,t)|^2 +\int_{B_{3r}(x_0)} K(x-y):v(y,t)\otimes v(y,t)\,dy
        \\&+\int_{\R^3\setminus B_{3r}(x_0)} (K(x-y)-K(x_0-y)):v(y,t)\otimes v(y,t)\,dy
}
in $L^{3/2}(B_{2r}(x_0) \times (0,T))$, where $K(x)= \nabla^2(\frac{1}{4\pi|x|})$, and

\item for any compact supported functions $w \in L^2(\R^3)^3$,
\EQ{\label{weak-continuity}
\text{the function}\quad
t \mapsto \int_{\R^3} v(x,t)\cdot w(x)\,dx \quad \text{is continuous on }[0,T).
}

\end{enumerate}
\end{definition}

The next lemma shows  a priori bounds for the
local energy solution, where the crucial part is proved by Lemari\'e-Rieusset~\cite{LR}
and later revised in \cite{KS,JS13}.
The following version is deduced from \cite{KMT20, BT19}.
\begin{lemma}[a priori bound of local energy solutions]
 \label{apriori uloc}
Suppose  $(v,p)$ is a local energy solution to \eqref{NS} with divergence free initial data $v_0\in L^2_{\uloc}$.
For any $s,q>1$ with   $\frac 2{s}+ \frac 3{q} = 3$,
there exist  $C(s,q)>0$ and $c_{x_0}(t) \in \R^3$
such that
\begin{align}
\label{ineq.apriorilocal}
E_{\uloc}(t):=\esssup_{0\leq s \leq t}
\|v(t)\|_{L^2_{\uloc}}^2
+
\sup_{x_0\in \RR^3}\int_0^t\!\int_{B_1(x_0)} |\nabla  v|^2 &\le
2 \|v_0\|_{L^2_{\uloc}}^2,
\\
\label{p.apriori}
\sup_{x_0 \in \R^3}  \norm{p -c_{x_0}}_{L^s(0,t; L^q(B_1(x_0)))}
&\le C(s,q) \|v_0\|_{L^2_{\uloc}}^2
\end{align}
for $t\le T_{\uloc} :=\frac{c_0}{1+\|v_0\|_{L^2_{\uloc}}^4}
$
with a universal constant $c_0>0$.
Similar estimates with $B_1$ replaced by $B_r$ are valid.
\end{lemma}
We now recall the scaled version of the $\epsilon$-regularity theorem of Caffarelli-Kohn-Nirenberg \cite[Proposition 1]{CKN}. It is formulated in the present form in \cite{NRS,Lin}.
\begin{lemma}\label{CKN-Prop1}
There are absolute constants $\e_{\CKN}$ and $C_{\CKN}>0$ with the following property. Suppose $(v,p)$ is a suitable weak solution of (NS)
 in $B_r(x_0) \times (t_0-r^2,t_0))$, $r>0$, with
\[
\frac 1{r^{2}} \int^{t_0}_{t_0-r^2} \int_{B_{r}(x_0)} |v|^3
+ |p|^{3/2} dx\,dt \le \e_{\CKN},
\]
then $v  \in L^\infty(B_{\frac{r}{2}(x_0)} \times (t_0-\frac{r^2}{4},t_0))$ and
\EQ{\label{CKN-Prop1-eq1}
\norm{v}_{L^\infty(B_{\frac{r}{2}(x_0)} \times (t_0-\frac{r^2}{4},t_0))} \le \frac {C_{\CKN}} {r} .
}
\end{lemma}

We recall a useful Gronwall-type inequality from Bradshaw and Tsai \cite[Lemma 2.2]{BT19}.
\begin{lemma}\label{gronwall}
Suppose $f \in L^\infty_{\loc}([0, T_0); [0,\infty))$
(which may be discontinuous) satisfies, for some
 $a$, $b>0$, and $m \ge 1$,%
$$
f(t) \le a + b
\int^t_0
(f(s) + f(s)^m)ds \qquad \mbox{for}\  t \in (0, T_0),
$$
then we have $f(t) \le 2a$ for $t \in (0, T)$
with
$$
T=\min\left(T_0,\, \frac{C}{b(1+a^{m-1})}\right),
$$
where $C$ is a universal constant.
\end{lemma}

Finally we also show the following elementary bound for the scaled energy.
\begin{lemma}
\label{lemmaA}

Assume that
 $f \in L^2_{\loc}(\R^3)$ and let $N_R=N_R(f):= \sup_{R\le r \le 1}\frac{1}{r}\int_{B_r}|f|^2$
 for some $R\in (0,\frac12]$. If  $\delta \ge 2N_R$, then for any $x_0 \in B_{\frac12}$
  we have
  \begin{equation}
\label{N_R}
\sup_{R(x_0)\le r\le {1-|x_0|} }
\frac{1}{r}\int_{B_r(x_0)}|f|^2
\le \delta,
\end{equation}
with $R(x_0)=\max\{{\frac R2}, \frac{2N_R|x_0|}{\delta} \}$.
\end{lemma}
\begin{proof}
Assume that $r \in [R(x_0), 1-|x_0|]$. {If $1/2<r\le 1-|x_0|$,
\[
\frac1{r}\int_{B_r(x_0)} |f|^2
\le \frac{1}{r}\int_{B_{1}(0)} |f|^2
\le \frac{1}{r} N_R \le \delta.
\]
If $|x_0|\le r\le 1/2$,}
\eqref{N_R} clearly holds by the following estimate:
$$
\frac1{r}\int_{B_r(x_0)} |f|^2
\le \frac{1}{r}\int_{B_{2r}(0)} |f|^2
\le 2N_R \le \delta,
$$
where we have used $|x_0| +r \le 2r$. 
 used $|x_0| +r \le 2r$. 
{Finally, if $R(x_0)\le r <|x_0|$ (this case is empty if $|x_0|<R/2\le R(x_0)$)}, we have $|x_0|+r <2|x_0|$,
and hence
\begin{align*}
\frac1{r}\int_{B_r(x_0)} |f|^2
&\le
\frac1{r}\int_{B_{2|x_0|}(0)} |f|^2
\le
\frac{2|x_0|}{r}N_R.
\end{align*}
The right hand side is bounded by $\delta$ since
$r \ge \frac{2N_R|x_0|}{\delta}$.
Therefore we have verified \eqref{N_R}.
\end{proof}

%

%
\section{Main results}
\label{main results}

In this section we prove Theorem \ref{thm:local}
regarding local regularity of suitable weak solutions.
It is obtained as a consequence of the following theorem;
see Remark \ref{0510} below.
\begin{theorem}
\label{local data}
Let $(v,p)$ be a suitable weak solution in $B_2\times (0, T_0)$,
$T_0>0$, associated with the initial data $v_0$ in the sense that
$\lim_{t\rightarrow 0+}\|v(t)-v_0 \|_{L^2(B_2)}=0$.
There are absolute constants $c$, $C \in (1,\infty)$
such that the following holds true.
\\
{\rm (i)}\, Let $N_R= \sup_{R< r \le 1}\frac{1}{r}\int_{B_r}|v_0|^2<\infty$, $R \ge 0$.
For any $M \in (0,\infty)$ and $\delta \in [ 5N_R,\infty)$, 
there exists $T=T(M,\delta,T_0) \in (0,T_0]$ such that
if
\begin{equation}
\label{03283}
\norm{v}^2_{L^{\infty}(0,T;L^2(B_2))} 
+ 
\|\nabla v\|_{L^2(B_2 \times (0,T))}^2
+
\norm{p}_{L^{\frac32}(B_2 \times (0,T))}\le M,
\end{equation}
then {$E_r(t)$ defined by \eqref{scaled energy}} and $(v,p)$ 
satisfy
\begin{align}
E_r(t) \le \delta \qquad \mbox{for} \ t \in 
(0,\min \{ \lambda r^2, T \}],
\qquad {\text{for all } r \in [R,r_1]}
\label{est:03302}
\end{align}
if $R \le r_1$,
and
\begin{align}
 \frac{1}{r^2}
\int_0^{\lambda r^2}\!\!\!\int_{B_r}\abs{v}^3+|p|^{\frac32}dxdt
&\le C(\delta +\delta^{\frac32})
\qquad \text{for all} \ r  \in [R,  r_2]
\label{0423}
\end{align}
if $R \le r_2$,
where $T$, $\lambda$, $r_1$, and $r_2$ are given by
$$
{T=\min\bke{T_0,\frac{c\min\{1,\delta^{12}\}}{1+M^{18}}}}, 
\qquad
\lambda=\frac{c}{1+\delta^2}, 
\qquad 
r_1 = \min\Big\{\frac{c\delta}{M^{\frac32}}, {\frac13} \Big\},
\qquad 
r_2:=\min \Big\{\sqrt{\frac{T}{\lambda}}, r_1  \Big\}.
$$ 
{\rm (ii)}\,There exists $\epsilon_*>0$ such that if $N_R \le \epsilon_*$ and $R^2 \le \widetilde T$, then $v$ is regular in the set
\[
\Pi=\bket{(x,t) \in B_{\frac12} \times [R^2,\widetilde{T}] \, : \,\,
t \ge cN_R^2\abs{x}^2}
\]
and satisfies
\[
\abs{v(x,t)}\le \frac{C}{\sqrt t} \qquad  \mbox{for} \ (x,t) \in \Pi
\]
with $\widetilde{T}=\min \{T_0, \frac{c}{1+M^{18}}\}$ and 
absolute constants $c$, $C>0$.
\end{theorem}
\begin{remark}
\label{0510}
(1) Theorem \ref{local data} sharpens and generalizes 
Theorem \ref{thm:local} in the following sense:
(a) The assumption \eqref{03283} is weaker than \eqref{0509}. More importantly, 
(b) it treats the case of general $R \ge 0$ 
so that the regular set $\Pi$ is more specifically characterized. 
It should be emphasized that if $R>0$ the assumption $N_R \le \epsilon_*$ in (ii)
 does not require  scale critical regularity at the origin. 
 
{(2) In Part (i) (and also Theorem \ref{large data} (i)), $\delta$ can be large, although it is small in the rest of the paper in particular (ii).
We hope it might be useful for some other applications.
} 
\end{remark}

\begin{proof}(i)~For the convenience, we let 
\begin{align*}
\mathcal{E}_{R,r_1}(t)&:=\sup_{R\le r\le r_1} E_r(t)
\end{align*}
with the constant $r_1$ to be specified later.
The local energy inequality \eqref{CKN-LEI} 
with a test function
$\varphi \in C^{\infty}_0(B_{2r})$ such that $0\le \varphi \le 1$
in $B_{2r}$ with $\varphi=1$ in $B_r$
and $\|\nabla^k \varphi \|_{L^\infty} \le C_kr^{-k}$
leads to\footnote{It is not difficult to see that 
we may take time-independent 
test functions in the local energy inequality provided the solution is continuous 
at $t=0$ in $L^2_{\loc}$. See, e.g., \cite[Remark\,1.2]{MMP}.}.
\begin{align*}
\int |v(t)|^2\varphi^2 dx
+
2\int^t_0 \!\!\int |\nabla v|^2 \varphi^2
&\le
\int |v_0|^2\varphi^2 dx+
\int^t_0\!\!\int |v|^2\Delta(\varphi^2) +(|v|^2+p)v
\cdot \nabla \varphi^2 dxds
\\
&\le
\int_{B_{2r}}\!\! |v_0|^2 dx+
\frac{C}{r^2} \int^t_0\!\!\!\int_{B_{2r}}\!\!|v|^2
+
\frac{C}{r}\int^t_0\!\! \int_{B_{2r}}\!\!|v|^3+|p|^{\frac32}dxds.
\end{align*}
This
implies
\begin{align}
E_r(t)
\le
&2N_R
+\frac{C}{r^3}\int_0^t\int_{B_{2r}} \abs{v}^2+ \frac{C}{r^2}\int_0^t\int_{B_{2r}} \abs{v}^3+\frac{C}{r^2} \int_0^t\int_{B_{2r}}\abs{p}^{\frac{3}{2}}
\notag
\\
=:
&2N_R+I_{lin} +I_{nonlin}+ I_{pr}.
\label{est:0303}
\end{align}
For any $\rho\in {(3r, 1]}$ we decompose the pressure as $p=\tilde{p}+p_h$ with
\begin{align*}
\tilde{p}(x):=\mbox{p.v.}\int  K(x-y) \xi(y)(v\otimes v)(y)dy - {\frac 13(\xi |v|^2)(x)},
\end{align*}
where $\xi$ is a smooth cut-off function with $\xi=1$ in $B_{\rho}$ and supported in $B_{2\rho}$.
Since $p_h$ is harmonic in $B_\rho$, 
\begin{align*}
\int_{B_{2r}}\abs{p}^{\frac{3}{2}}
\le &C\int_{B_{2r}}\abs{\tilde{p}}^{\frac{3}{2}}+C\int_{B_{2r}}\abs{p_h}^{\frac{3}{2}}
\\
\le &C\int_{B_{2r}}\abs{\tilde{p}}^{\frac{3}{2}}+
C(\frac{r}{\rho})^{3}\int_{B_{\rho}}\abs{p_h}^{\frac{3}{2}}
\\
\le &C\int_{B_{2r}}\abs{\tilde{p}}^{\frac{3}{2}}+
C(\frac{r}{\rho})^{3}\int_{B_{\rho}}\abs{\tilde{p}}^{\frac{3}{2}}+
C(\frac{r}{\rho})^{3}\int_{B_{\rho}}\abs{p}^{\frac{3}{2}}
\end{align*}
By the Calder\'on-Zygmund estimate we see{
\EQ{
\int_{B_{2r}}\abs{p}^{\frac{3}{2}} \le C\int_{B_{2\rho}}\abs{v}^3
+
C(\frac{r}{\rho})^{3}\int_{B_{\rho}}\abs{p}^{\frac{3}{2}}.
\label{est:0304}
}}

We now divide the proof into
two cases. Let $r_0$ be a constant satisfying $
0\le r_0\le r_1/3$ to be fixed later.
\\
{\underline{\bf Case I:~$R\le r\le r_0$}.} {If $r_0<R<r_1$, this case is empty, which is fine.} 
Noting that  $2r \in [R,r_1]$, we easily observe that
\begin{align}
\label{est:0324}
I_{lin} \le \frac{C}{r^2}\int^t_0 \mathcal{E}_{R,r_1}(s)ds,
\end{align}
and also by the interpolation inequality,
\begin{align}
\notag%
I_{nonlin} &\le \frac{C}{r^2}
\int_0^t\bke{\int_{B_{2r}} \abs{ \nabla v}^2}^{\frac{3}{4}}\bke{\int_{B_{2r}} \abs{ v}^2}^{\frac{3}{4}}
ds
+\frac{C}{r^2}\int^t_0\bke{\frac{1}{r}\int_{B_{2r}} \abs{v}^2}^{\frac{3}{2}}ds
\\
&\le
 \frac{\epsilon}{r}\int^t_0\int_{B_{2r}}|\nabla v|^2ds
+
\frac{C_{\epsilon}}{r^2} \int^t_0 \bke{\frac{1}{r}\int_{B_{2r}} \abs{v}^2}^3+\bke{\frac{1}{r}\int_{B_{2r}} \abs{v}^2}^{\frac{3}{2}}ds
\notag
\\
&\le 
\epsilon \mathcal{E}_{R,r_1}(t)
+\frac{C_{\epsilon}}{r^2}\int^t_0 \mathcal{E}_{R,r_1}^3(s)+\mathcal{E}_{R,r_1}^{\frac{3}{2}}(s)ds
\label{est:03043}
\end{align}
with some constant $\epsilon \in (0,1)$.
For the pressure term, integrating \eqref{est:0304} in time yields
\begin{align*}
I_{pr}
&\le
{\frac{C}{r^2}} \int_0^t\int_{B_{2\rho}}\abs{v}^3
+
\frac{C'r}{\rho}\frac{1}{\rho^2} \int_0^t\int_{B_{\rho}}\abs{p}^{\frac{3}{2}}
\end{align*}
with an absolute constant $C' > 1$.
Choose $\rho=10C'r$ so that $\frac{C'r}{\rho} \le \frac1{10}$
and also let $r_0 = \frac{r_1}{20C'}$.
Noting that $2\rho =20C'r \le r_1$ and  \eqref{est:03043}, we see that
\begin{align}
I_{pr}
&\le
\frac{C}{r^2} \int_0^t\int_{B_{{20C'r}}}\abs{v}^{3}
+\frac{1}{10}\mathcal{E}_{R,r_1}(t)
\notag
\\
&\le
\epsilon \mathcal{E}_{R,r_1}(t)
+\frac{C_{\epsilon}}{r^2}\int^t_0 \mathcal{E}_{R,r_1}^3(s)+ { \mathcal{E}_{R,r_1}}^{\frac{3}{2}}(s)ds
+\frac{1}{10}\mathcal{E}_{R,r_1}(t).
\label{Mar2-1}
\end{align}
Hence applying \eqref{est:0324}%
-\eqref{Mar2-1} in \eqref{est:0303} with $\epsilon=\frac{1}{20}$,
 we obtain 
\begin{equation}
\label{est:0328}
\mathcal{E}_{R,r_0}(t)
\le \frac{2\delta}{5}+\frac{C}{R^2}
\int^t_0
\mathcal{E}_{R,r_1}^3(s)+ \mathcal{E}_{R,r_1}(s)ds
+\frac15\mathcal{E}_{R,r_1}(t).
\end{equation}
{\underline{\bf Case II:~$r_0\le r\le r_1$}.}
For $t \le T$ with $T$ specified later,
we estimate the right hand side of  \eqref{est:0303} with
the aid of $M$.
A straightforward estimate yields
\[
I_{lin} \le \frac{C}{r_0^3}\int_0^t\int_{B_{2}} \abs{u}^2\le \frac{Ct}{r^3_0}M\le \frac{\delta}{10},%
\]
provided $t\le  \frac{c\delta r^{3}_0}{M}$ with a small absolute 
constant $c \in (0,1)$.
In the similar way to \eqref{est:03043} we have
\begin{align}
I_{nonlin}
\le& \frac{C}{r_0^2}\int_0^t\int_{B_2} \abs{v}^3
\notag
\\
\le&\frac{C}{r_0^2}
\int_0^t\bke{\int_{B_{2}} \abs{ \nabla u}^2}^{\frac{3}{4}}\bke{\int_{B_{2}} \abs{v}^2}^{\frac{3}{4}}ds
+\frac{C}{r_0^2}\int^t_0\bke{\frac{1}{r_0}\int_{B_{2}} \abs{v}^2}^{\frac{3}{2}}ds
\notag
\\
\le &
\frac{Ct^{\frac14}M^{\frac32}}{r_0^2}+\frac{CtM^{\frac32}}{r_0^\frac72}.
\label{est:03044}
\end{align}
Thus $I_{nonlin}$ is bounded by $\frac{\delta}{10}$
if  $t\le \min \{ \frac{c\delta^4 r^{8}_0}{M^6}, \frac{c\delta r_0^{\frac72}}{{M^{\frac32}}} \}$ with a suitable absolute constant $c \in (0,1)$.
Concerning the pressure, we choose $\rho=1$ {(using $3r_1 \le 1=\rho$)} in \eqref{est:0304} and apply \eqref{est:03044} to get
\begin{align}
I_{pr}
\le &
\frac{C}{r^2}\int^t_0 \int_{B_{2r}}\abs{\tilde{p}}^{\frac{3}{2}}
+
Cr \int^t_0\int_{B_{1}}\abs{p}^{\frac{3}{2}}
\notag
\\
 \le &
{ \frac{C}{r_0^2}} \int^t_0 \int_{{B_{1}}} \abs{\tilde{p}}^{\frac{3}{2}}
+
Cr_1 M^{\frac32}
\notag
\\
\le &
\frac{C}{r_0^2} \int^t_0 \int_{B_{2}}\abs{v}^3
+Cr_1 M^{\frac32}
\notag
\\
\le &
\frac{Ct^{\frac14}M^{\frac32}}{r_0^2}
+
\frac{CtM^{\frac32}}{r_0^\frac72}
+
Cr_1M^{\frac32}
\le  \frac{\delta}{10},
\label{est:03284}
\end{align}
provided $t\le \min \{ \frac{c\delta^4 r^{8}_0}{M^6}, \frac{c\delta r_0^{\frac72}}{{M^{\frac32}}} \}$
and $r_1 \le \min \{ \frac{c\delta}{M^{\frac32}},{\frac13}\}$
with an absolute constant $c>0$.
\\
Making use of  these estimates in \eqref{est:0303},
we obtain that
\begin{equation}
\label{est:03282}
\mathcal{E}_{r_0,r_1}(t) \le \frac{\delta}{2}  \qquad \mbox{if}
\ \
t \le \min \Big\{ \frac{c\delta r_0^3}{M},
\frac{c\delta^4 r^{8}_0}{M^6}, \frac{c\delta r_0^{\frac72}}{{M^{\frac32}}} \Big\}
\ \ \mbox{and} \ \  r_1 \le \min \left\{ \frac{c\delta}{M^{\frac32}},{\frac13} \right\}.
\end{equation}
Combining \eqref{est:0328} and \eqref{est:03282}, we see
\begin{equation}\label{0621}
{\mathcal{E}_{R,r_1}}(t)
\le \frac{\delta}{2}+\frac{C}{R^2}
\int^t_0
\mathcal{E}_{R,r_1}^3(s)+ \mathcal{E}_{R,r_1}(s)ds, \quad {0<t<T,}
\end{equation}
{where
\begin{equation}\label{T.defn}
T=\min\bke{T_0,\frac{c\min\{1,\delta^{12}\}}{1+M^{18}}}.
\end{equation}
Note} that 
$T
\le
 \min \Big\{ \frac{c\delta r_0^3}{M},
\frac{c\delta^4 r^{8}_0}{M^6}, \frac{c\delta r_0^{\frac72}}{{M^{\frac32}}} \Big\}$
with a suitable small constant $c>0$
upon the choice of $r_0=\frac{r_1}{20C'}$ with $r_1 = \min \left\{ \frac{c\delta}{M^{\frac32}},{\frac13} \right\}$.
The inequality \eqref{0621} is also true if $R$ is replaced by any $r \in [R,r_1]$ with the same $\de$. 
Therefore we may invoke Gronwall Lemma \ref{gronwall} 
with $f(t)={\mathcal{E}_{r,r_1}}(t)$ with
$a=\frac{\delta}{2}$, $b=\frac{C}{r^2}$, and $m=3$
to see that
\begin{equation}
\mathcal{E}_{r,r_1}(t) \le \delta \qquad \mbox{for} \ t \in 
\big(0,\min \big \{ T, \lambda r^2 \big\}\big] 
\quad \mbox{and} \quad  r \in [R,r_1],
\qquad  \lambda=\frac{c}{1+\delta^2}.
\end{equation}
This proves \eqref{est:03302}.

In the similar way to \eqref{est:03043}, 
we see that if $\lambda r^2 \in (0,T]$ and $r\in [R,r_1]$,
\begin{align*}
\frac1{r^2}\int^{\lambda r^2}_0 \int_{B_{r}}|v|^3
&\le
\frac{C}{r^2} \int_0^{\lambda r^2}\bke{\int_{B_{r}} \abs{ \nabla v}^2}^{\frac{3}{4}}\bke{\int_{B_{r}} \abs{v}^2}^{\frac{3}{4}}ds
 +\frac{C}{r^2}\int^{\lambda r^2}_0\bke{\frac{1}{r}\int_{B_{r}} \abs{u}^2}^{\frac{3}{2}}ds
\\
&\le C(\lambda^\frac14 +\lambda)E_r(\lambda r^2)^{\frac32}.
\end{align*}
Hence taking $r_2:=\min\{\sqrt{\frac{T}{\lambda}},r_1\}$,
we have from \eqref{est:03302} and $\lambda \le 1$ that
\begin{align*}
\frac1{r^2}\int^{\lambda r^2}_0 \int_{B_{r}}|v|^3
&\le
C \delta^{\frac32} \qquad \mbox{for} \ r \in [R, 
r_2]. 
\end{align*}
This together with the pressure bound
$$
\frac1{r^2}\int^{\lambda r^2}_0 \int_{B_{r}}|p|^\frac32
\le {E_r(\lambda r^2)}
\le 
C\delta
$$
leads to \eqref{0423} as desired.

\medskip

(ii)~In order to show the statement (ii), we claim
that there exist $\epsilon_*$ and $c>0$ such that
if $N_R < \epsilon_*$, then
for any $x_0 \in B_{\frac12}$ and 
$r \in [\max
\{R, cN_R|x_0|\}, {c(1-|x_0|)r_2}]$,
\begin{equation}
\label{ckn}
\frac{1}{r^2}\int_0^{r^2}
\!\!\!\int_{B_r(x_0)}\abs{v}^3 +\abs{p}^\frac32 dxdt
\le
\epsilon_{\CKN}
\end{equation}
holds. {Here $\epsilon_{\CKN}$ is the small constant in Lemma \ref{CKN-Prop1}.}
To this end, we first note that
for  each $\eta \ge 2N_R$
and $x_0 \in B_{1/2}$, Lemma \ref{lemmaA} implies
\begin{equation*}
\sup_{R(x_0)\le r\le \rho}
\frac{1}{r}\int_{B_r(x_0)}|v_0|^2dx \le \eta, \quad {\rho = 1-|x_0|},
\end{equation*}
with $R(x_0)=\max \big\{R/2, \frac{2N_R}{\eta}|x_0| \big\}$.
Let
$v_{x_0}(x,t)=\rho v(x_0+\rho x, \rho^2 t)$, $p_{x_0}(x,t)=\rho^2 p(x_0+\rho x, \rho^2 t)$ and $\delta=5\eta$.  Since $(v_{x_0},p_{x_0})$ solves \eqref{NS} in $B_2(0) \times (0,\rho^{-2}T_0)$, corresponding to $(v,p)$ in $B_{2\rho}(x_0)\times (0,T_0)$,
 and $1/2 \le \rho \le 1$,
\[
\norm{v_{x_0}}^2_{L^{\infty}(0,\rho^{-2}T_0;L^2(B_2))} 
+ 
\|\nabla v_{x_0}\|_{L^2(B_2 \times (0,\rho^{-2}T_0))}^2
+
\norm{p_{x_0}}_{L^{\frac32}(B_2 \times (0,\rho^{-2}T_0))}\le CM,
\]
\[
\sup_{\rho^{-1} R(x_0)\le r\le 1}
\frac{1}{r}\int_{B_r(0)}|v_{x_0}(x,0)|^2dx = \sup_{R(x_0)\le r\le \rho}
\frac{1}{r}\int_{B_r(x_0)}|v_0|^2dx \le \frac \de 5,
\]
by \eqref{0423} and  \eqref{T.defn} we get
\[
\sup_{\rho^{-1} R(x_0)\le r \le r_2 } \frac{1}{r^2}
\int_0^{\lambda r^2}\!\!\!\int_{B_r(0)}|v_{x_0}|^3+|p_{x_0}|^{\frac32}
\le C(\delta+\delta^{\frac32}).
\]
Here  $r_2=\min \big(\sqrt{\frac{T'}{\lambda}},r_1\big)$, with %
\[
T'=\min\bke{\rho^{-2} T_0,\frac{c\min\{1,\delta^{12}\}}{1+M^{18}}}, \quad
r_1 = \min \left( \frac{c\delta}{M^{\frac32}},\frac13 \right),
\]
with a smaller constant $c$. Note $T'$ differs from $T$ in \eqref{T.defn} by the factor $\rho^{-2}$ for $T_0$.
This implies
\begin{align}
\sup_{R(x_0) \le r \le \rho r_2 } \frac{1}{\lambda r^2}
\int_0^{\lambda r^2}\!\!\!\int_{B_{\sqrt\lambda r}(x_0)}\abs{v}^3+|p|^{\frac32}
\le \frac{C(\delta+\delta^{\frac32})}{\lambda}
\le C(1+\delta^2)(\delta+\delta^{\frac32}).
\label{0418}
\end{align}
Take a constant $\delta_0>0$ so small that $C(1+\delta^2_0)(\delta_0+\delta_0^{\frac32})\le \epsilon_{\CKN}$.
We now assume that $v_0$ satisfies $N_R \le \delta_0/10$. 
Then we may choose $\delta=\delta_0$ since 
$\delta_0 \ge 10N_R$.
With this choice and with 
$\lambda_0=\lambda(\delta_0)$, \eqref{0418} shows
\eqref{ckn} holds for
$\lambda_0^{\frac12} R(x_0) \le r \le \lambda_0^{\frac12}\rho r_2$.
This enables us to apply Lemma \ref{CKN-Prop1} for
$x_0 \in B_{\frac12}$ and  $t_0=r^2 \in [\lambda_0 \max\{R^2, \frac{CN_R^2|x_0|^2}{\delta_0^2}\}, \lambda_0\rho^2 r_2^2]$ 
to see
$$
|v(x_0,t_0)|\le \frac{C_{\CKN}}{r} = \frac{C_{\CKN}}{\sqrt{t_0}},
$$
and hence $v$ is regular at $(x_0,t_0)$. Since $r_2=\min\{\sqrt{\frac{T'}{\lambda_0}},r_1\}$ and $1/2 \le \rho \le 1$,
\[
\rho^2
\lambda_0 r_2^2=\rho^2\min( T', \la_0 r_1^2)
= \rho^2  \min \bke{\rho^{-2}T_0, \,  \frac{c\min\{1,\delta_0^{12}\}}{1+M^{18}}, \,\la_0 r_1^2}\ge  \min \bke{T_0, \,  \frac{c}{1+M^{18}}}.
\]
Thus $|v(x_0,t_0)|\le C_{\CKN} t_0^{-1/2}$ for $x_0 \in B_{1/2}$ and 
\[
\max\{R^2, cN_R^2|x_0|^2\} \le t _0\le \min\bke{T_0,  \,  \frac{c}{1+M^{18}}}.
\]
This shows Part (ii) of Theorem \ref{local data}.
\end{proof}

\begin{remark}
\label{0412}
If we assume slightly stronger regularity for the pressure as
 $p \in L^2(0,4;L^{\frac32}(B_2))$ and
\begin{equation}
\label{L}
\norm{v}^2_{L^{\infty}(0,T;L^2(B_2)) \cap L^2(0,T;H^1(B_2))}
+\norm{p}_{L^2(0,T;L^{\frac32}(B_2))} \le L
\end{equation}
for some $L>0$ instead of \eqref{03283},
then we have a bound for $T>0$ in terms of $L$:
Inserting the estimate
$$
\int^t_0\int_{B_1} |p|^{\frac32} \le t^{\frac14}L^{\frac32}
$$
 in \eqref{est:03284} we have 
$$
I_{pr} \le \frac{\delta}{10}
\qquad  
\mbox{for} \ t \le \frac{c\delta}{L^6} \ 
\mbox{and} \ r_1=\frac13
 $$ with a suitable constant $c>0$.
Inspecting the proof we easily verify we can replace $T=T(M,\delta)$
by  $T'=T'(L,\delta)=\frac{c\min \{ 1,\delta^4 \}}{1+L^6}$.
Note that $T'(L,\delta)$ is bigger than $T(M,\delta)$
for sufficiently large $L=M$. This will be used in the next theorem.
\end{remark}

\medskip

One implication of Theorem \ref{local data} is that for $L^2_{\uloc}$ initial data in $\R^3$, 
similar conclusions hold for the local energy solutions:
\begin{theorem}
\label{large data}
Let $v$ be a local energy solution of \eqref{NS} in $\R^3 \times (0,T_0)$ associated with 
initial data  $v_0 \in L^2_{\uloc}$, $N_R= \sup_{R< r \le 1}\frac{1}{r}\int_{B_r(0)}|v_0|^2<\infty$. The following hold true.
\\
{\rm (i)}\ For $R \in [0,\frac13]$,
let $\delta \ge 5N_R$. Then 
\begin{align}
E_r(t)
&\le \delta \quad  \mbox{for} \ t \le \min \{ \lambda r^2, T_1 \}
\qquad {\text{for all } r \in [R,\frac13}],
\label{A_r}
\end{align} holds with constants given by
$$\lambda=\frac{c}{1+\delta^2} \le 1,
 \qquad T_1=  \min \left ( T_0,\,
 \frac{c\min(1,\delta^4) }{1+\|v_0\|_{L^2_{\uloc}}^{12}}
 \right).$$
Moreover for any $r \in [R,  R_1]$ with $R_1:=\min \{\sqrt{\frac{T_1}{\lambda}}, \frac13  \}$,
there exists $c_2(t)$ such that
\begin{align}
 \frac{1}{r^2}
\int_0^{\lambda r^2}\!\!\!\int_{B_r}\abs{v}^3+|p-c_2(t)|^{\frac32}dxdt
&\le C(\delta+\delta^{\frac32}).
\label{est:0305}
\end{align}
Here $c$ and $C>0$ are absolute constants.
\\
{\rm (ii)}\,There are absolute constants $\epsilon_*$, $c_0$, $c_1$, and $C_2>0$ such that the following holds.
If $N_R \le \epsilon_*$
for some $R \le \min\{\sqrt{T_2}, \frac13 \}$
with $T_2=\min(T_0,\, \frac{c_1}{1+\|v_0\|_{L^2_{\uloc}}^{12}}) $ , then $v$
is regular in the set
 $$
 \Pi:=\left\{ (x,t) \in B_{\frac12} \times [R^2,T_2];\, c_0{N}_R^2|x|^2\le t \right\}
 $$
 and  satisfies
\begin{align}
 |v(x,t)| &\le \frac{C_2}{\sqrt{t}} \qquad \mbox{for} \ (x,t)
\in \Pi.
\end{align}

\end{theorem}
\begin{proof}
By Lemma \ref{apriori uloc} with $(s,q)=(2,3/2)$,
\begin{equation*}
\esssup_{0\leq t \leq T_{\uloc} }\int_{B_2}|v|^2  \,dx +  \int_0^{T_{\uloc}}\!\!\int_{B_2} |\nabla  v|^2\,dx\,dt \le 2\|v_0\|_{L^2_{\uloc}}^2,
\end{equation*}
\begin{equation*}
\norm{p -c_2(t)}_{L^2(0,T_{\uloc}; L^{\frac32}(B_2))} \le C(2,3/2)\|v_0\|_{L^2_{\uloc}}^2.
\end{equation*}
The pair
 $(v, p-c_2(t))$ is a suitable weak solution in $B_2\times (0,\min(T_0,T_{\uloc}))$ and it
satisfies the assumptions in \eqref{L}
with $L=(2+C(2,3/2))\|v_0\|_{L^2_{\uloc}}^2$.
Hence from Remark \ref{0412}, we see \eqref{est:03302} holds
with $T$ replaced by
\[
T'= \min \bket{T_0,\,\frac{c\min \{ 1,\delta^4 \}}{1+\|v_0\|_{L^2_{\uloc}}^{12}}}
\le  \min \bket{T_0, T_{\uloc},\,\frac{c\min \{ 1,\delta^4 \}}{1+L^6}},
\]
with $c>0$ sufficiently small.
Thus we obtain \eqref{A_r} as desired. 
We omit the proof
of the remaining claim, since it can be proved in exactly the
same way as in Theorem  \ref{local data}.
\end{proof}
\begin{remark} 
In the appendix, we provide an alternative proof of
Theorem \ref{large data} based on the pressure decomposition \eqref{decomposition} 
instead of the observation in Remark \ref{0412}, 
which seems to be of independent interest.
\end{remark}

\section{Applications}
\label{applications}
\subsection{Regular set}

We now show Theorem \ref{0311}, which contains 
Theorem \ref{thm:0430} as a special case and is useful
for further applications.
To this end, define $\dot{N}_R$ by
\begin{equation*}
\dot{N}_R:=\sup_{r>R}\frac{1}{r}\int_{B_r(0)} \abs{v_0}^2  \qquad (R\ge 0).
\end{equation*}

\begin{theorem}\label{0311}
Let $(v,p)$ be a local energy solution in $\R^3 \times (0,\infty)$
for the initial data $v_0 \in L^2_{\uloc}(\R^3)$. Let $\epsilon_*$ and $c_0$ be the absolute constants in Theorem \ref{large data} (ii).
The following statements hold:
\\
{\rm (i)}\
If $v_0$ satisfies \eqref{decay'}, i.e.,
\begin{equation*}
M_1:=\sup_{x_0\in \R^3}\sup_{r\ge 1}\frac{1}{r}\int_{B_r(x_0)} \abs{v_0}^2 dx<\infty,
\end{equation*}
and if
\begin{equation}
\label{03112}
\dot{N}_R\le \epsilon_* \qquad \mbox{for \ some} \ R\ge 0,
\end{equation}
then $v$ is regular in the set
$\bket{(x,t) \in \R^3\times (0, \infty)\,:\,
\max \{R^2, c_0\dot{N}_R^2 \abs{x}^2 \} \le t}$.
\\
{\rm (ii)}\,Suppose $v_0 \in L^{2,-1}(\R^3)$. For any $0<\de\le \e_*$, there exist
positive constant $T(v_0,\de)$ such that $v$ is regular in the set
\begin{equation}
\label{0706}
\bket{(x,t) \in \R^3 \times (0,\infty):~c_0 \de^2 |x|^2\le t\le T(v_0,\de)}.
\end{equation}

 If $v_0 \in L^{2,-1}(\R^3)$ also satisfies \eqref{decay'}, then for any $\de \in (0,\e_*]$, there is $T'(v_0,\de)$ such that 
$v$ is regular in
\begin{equation}
\label{0707}
\bket{(x,t) \in \R^3 \times (0,\infty):~ \max(c_0\de^2 |x|^2, T'(v_0,\de)) \le t}.
\end{equation}
It is also regular in 
\begin{equation}
\label{0707b}
\{(x,t)\in \R^3 \times (0,\infty):~c_0\epsilon_*^2|x|^2\le t<r_*^2 T_2(M_2)  \},
\end{equation}
with $r_*:=\sup\{r>0  \mid \|v_0\|_{L^{2,-1}(B_r)}^{2} \le \epsilon_* \}$, $0<r_* \le \infty$, and $M_2= [(\max(1,\frac1{r_*})M_1]^{1/2}$. When $r_*=\infty$, i.e.,
$ \|v_0\|_{L^{2,-1}(\R^3)}^2 \le \epsilon_*$,
the regular set \eqref{0707b} has no time upper bound.

\end{theorem}
Theorem \ref{0311} is more general than Theorem \ref{thm:0430} since it
assumes $\dot{N}_R<\epsilon_*$  for some $R\ge 0$ while Theorem \ref{thm:0430} 
assumes $\dot{N}_0<\epsilon_*$.
The regular set in Part (i) does not depend on the value of $M_1$, although it needs $M_1<\infty$.
If $v_0 \in L^{2,-1}(\R^3)$ also satisfies \eqref{decay'}, $v$ is regular in the union of the two sets \eqref{0706} and \eqref{0707}.
These sets are inside the parabola $c_0\de^2|x|^2 \le t$ with arbitrarily small $\de>0$.
The set \eqref{0707} is similar to the \emph{eventual regularity} result given in \cite[Theorem 1.4]{BKT}, which contains an additional a priori bound assumption on the solution,
available for solutions constructed in \cite{BKT} for $v_0$ in certain Morrey type space. That bound plays a similar role as the assumption \eqref{decay'} for $v_0$.

\begin{proof}
(i)  We first consider the case $R>0$.
Let $u_0(x)=\lambda v_0(\lambda x)$,
$u(x,t)=\lambda v(\lambda x, \lambda^2 t)$ for $\lambda>2R$. By the assumption, we have
\[
\sup_{\frac{R}{\lambda}\le r<1}\frac{1}{r} \int_{B_r} \abs{u_0}^2\le \dot{N}_R \le \epsilon_* ,
\]
and
\begin{equation}
\label{0707c}
\norm{u_0}_{L^2_{\rm uloc}} = \sup_{x_0\in\R^3} \bke{\frac1{\la}
\int_{|x-x_0|< \la} |v_0(x)|^2 dx}^{\!1/2}
\le \max\bke{\frac1{R}\norm{v_0}_{L^2_\uloc}^2, M_1}^{\!1/2}=:C_R,
\end{equation}
with
$C_{R}>0$ independent of $\lambda$.
By Theorem \ref{large data}\,(ii), there exists $T_2'=T_2'(C_R)$
independent of $\lambda$ such that
 $u$ is regular if $
\max \{R^2/\lambda^2, c_0\dot{N}_R^2 \abs{x}^2\} \le t\le T_2'$
for $\lambda \ge \frac{R}{\sqrt{T_2'}}$.
Scaling back we see that
$v$ is regular if  $\max \{R^2, c_0\dot{N}_R^2 \abs{x}^2\} \le t \le\lambda^2 T_2'$.
Since $\lambda>\max\{2R, \frac{R}{\sqrt{T_2'}} \}$ is arbitrary, $v$ is regular in the 
set  $\{(x,t):\,$$\max \{R^2, c_0\dot{N}_R^2 \abs{x}^2\}\le t \}$. This proves the case $R>0$.
For the case $R=0$, by $\dot{N}_r \le \dot{N}_0$, the above argument shows $v$ is regular in
$\{(x,t):\,$$\max \{r^2, c_0\dot{N}_r^2 \abs{x}^2\}\le t \}$.
Since $r>0$ is arbitrary, we conclude the proof.

\medskip
(ii) Suppose now $v_0 \in L^{2,-1}(\R^3)$. For any $\de \in (0,\e_*]$, there exists $R_0>0$ such that
$$
\sup_{0<r\le R_0} \frac{1}{r}\int_{B_r} \abs{v_0}^2 dx
\le
\int_{B_{R_0}} \frac{|v_0|^2}{|x|}dx \le \de.
$$
Let $u_0(x)=\lambda v_0(\lambda x)$ and $u(x,t)=\lambda v(\lambda x,\lambda^2 t)$ with $\lambda=R_0$.
We easily see $u_0 \in L^2_{\uloc}$ and
$\sup_{0<r\le 1} \frac{1}{r}\int_{B_r} \abs{u_0}^2 dx \le \de\le \epsilon_*$. By Theorem \ref{large data} (ii),
$u$ is regular if $c_0\delta^2|x|^2\le t\le T_2(u_0)$.  Hence $v$ is regular in the set $\{(x,t)\mid c_0\delta^2|x|^2\le t\le R_0^2T_2(u_0)\}$.
Note $T_2(u_0)$ depends on both $R_0$ and $\norm{v_0}_{L^2_\uloc}$ and goes to zero rapidly as $R_0 \to 0$.

Suppose now $v_0 \in L^{2,-1}$ also satisfies \eqref{decay'}.
There is $\rho>0$ such that $ \int_{\R^3\setminus B_\rho } \frac{1}{|x|} |v_0|^2\le \de/2$. Let
$R_1=\max(\rho, \frac {2}{\de}  \int_{B_\rho }  |v_0|^2)$. 
For any $r \ge R_1$, we have
\[
\frac 1r \int_{B_r} |v_0|^2 \le \frac 1r \int_{B_\rho} |v_0|^2  + \int_{B_r\setminus B_\rho } \frac{1}{|x|} |v_0|^2
\le \frac \de2 + \frac \de2 .
\]
Thus $\dot N_{R_1}\le\de$. By Part (i), $v$ is regular in the set $\bket{ \max(R_1^2, c_0 \de^2 |x|^2)\le t}$.

The remaining statement follows by
choosing $\de=\epsilon_*$, $\lambda=R_0\to r_*$, and noting $\norm{u_0}_{L^2_\uloc} \le M_2$ using \eqref{0707c}.
\end{proof}

\begin{remark}
\label{0321}
By a standard localization argument and using the a priori estimate derived in this theorem,
it is not difficult to construct a suitable weak solution (but not necessarily
a local energy solution) satisfying the conclusion
in (i) for the data satisfying \eqref{03112} without assuming
neither conditions $v_0 \in L^2_{\uloc}$ nor \eqref{decay'}.
However, since such a result is already studied in \cite{BKT} by
a slightly different approach, we do not discuss this fact in details here.
\end{remark}
\begin{remark}[Initial regular set for general data]
\label{0408}
Following \cite[Comment 4 after Corollary 1.2]{KMT20}, define
\[
\rho(x)= \rho(x;v_0) :=\sup \bket{ r>0: v_0 \in L^2(B_r(x)), \frac1r \int_{B_r(x)} |v_0|^2 \le
\epsilon_*}.
\]
Let $\rho(x)=0$ if such $r$ does not exist, and let $\rho(x)=\infty$ if
$ \sup_{r>0}\frac{1}{r}\int_{B_r(x)} |v_0|^2 \le \epsilon_*$. We also set
\[
T(x) =\sup_{\rho < \rho(x)}\frac{c_2\rho^2}{1+\rho^{-6}\|v_0\|_{L^2_{uloc,\rho}}^{12}}  \in [0,\infty] \qquad \mbox{with} \ \ 
\|v_0\|_{L^2_{uloc,\rho}}:= \sup_{x_0 \in \R^3}
\|v_0\|_{L^2(B_\rho(x_0))}. 
\]
For each $x_0 \in \R^3$, applying Theorem \ref{large data} (ii) to %
$$
u(y,t)=\lambda v(\lambda y+x_0,\lambda^2t)
$$
with $\lambda=\rho(x_0)$ we see that any local energy solution $v$  is regular in the set
\[
\bket{(x,t)\in \R^3 \times (0,\infty):\, 0< t < T(x)}.
\]
This improves  \cite[Comment 4 after Corollary 1.2]{KMT20} because we replace local $L^3$ norm by local Morrey-type norm in the 
definition of $\rho(x)$.
\end{remark}

\medskip

We next apply Theorem \ref{0311} to prove Corollary \ref{thm:weighted}.

\begin{proof}[Proof of Corollary \ref{thm:weighted}]
(i)~By the assumptions, we have $v_0 \in L^2(\R^3)$, 
and hence it satisfies
\eqref{decay'} since
$$\sup_{x_0 \in \R^3} \frac{1}{r}\int_{B_r(x_0)} \abs{v_0}^2
\le \frac{1}{r}\|v_0\|_{L^2}^2
$$
This estimate also implies
\begin{equation}
\label{0402}
\sup_{x_0 \in \R^3,\,r\ge  R_*} \frac{1}{r}\int_{B_r(x_0)} \abs{v_0}^2
\le \epsilon_*
\qquad  \mbox{with} \  R_*:=\frac{\|v_0\|_{L^2}^2}{\epsilon_*}.
\end{equation}
Therefore applying Theorem \ref{0311}\,(i) to 
$v_{x_0}(x,t):=v(x+x_0,t)$
  for each $x_0 \in \R^3$, we see
\[
\bket{(x,t) \in \R^3\times (0, \infty)\,:\,
\max \{R_*^2, c_0\dot{N}_{R_*}^2 \abs{x-x_0}^2 \} \le t}
\]
is the regular set of $v$. Since $x_0 \in \R^3$ is arbitrary,
this shows that
$v$ is regular for $t \ge R_*^2=\|v_0\|_{L^2}^4/\epsilon_*^2$.

We next show that
 there exists
$M=M(\|v_0\|_{L^{2,\alpha}})$ such that if $|x_0| \ge 2R_*$,
\begin{equation}
\label{03152}
\sup_{M|x_0|^{-\alpha}\le r\le R_*} \frac{1}{r}\int_{B_r(x_0)}|v_0|^2 \le \epsilon_*.
\end{equation}
Indeed since $r \le R_* \le |x_0|/2$, we see
\begin{align*}
\frac{1}{r}\int_{B_r(x_0)} \abs{v_0}^2
=
\frac{1}{r}\int_{B_r(x_0)} \frac{|x|^{\alpha}|v_0|^2}{|x|^{\alpha}}
\le
\frac{2^\al}{r|x_0|^{\alpha}}\|v_0\|_{L^{2,\alpha}}^2,
\end{align*}
from which \eqref{03152}  follows with
$M=\frac{{2^\al} \|v_0\|^2_{L^{2,\alpha}}}{\epsilon_*}$.
Combining this with \eqref{0402} implies
\[
\sup_{M|x_0|^{-\alpha} \le r \frac{1}{r}\int_{B_r(x_0)} } \abs{v_0}^2\le
\epsilon_*
\qquad
\mbox{provided} \ |x_0|\ge 2R_*.
\]
Then if $|x_0| \ge 2R_*$, we may apply Theorem \ref{0311}\,(i)
for $v_{x_0}$ with $R=M|x_0|^{-\alpha}$, to see that
$v$ is regular at $(x_0,t)$ for $t \ge M^2|x_0|^{-2\alpha}$.
Since $v$ is also regular for $t \ge R_*^2=\|v_0\|_{L^2}^4/\epsilon_*^2$, 
this finishes the proof of  (i) of Corollary \ref{thm:weighted}, with $K=\max(M^2,4^\al R_*^{2+2\al})$.
Note that $K|x|^{-2\al}\ge R_*^2$ when $|x|\le 2R_*$.

\medskip

(ii)~Now $\al\in (-1,0)$. We use similar approach as above: If $r \ge |x_0|/2$,
\begin{align*}
\frac{1}{r}\int_{B_r(x_0)} \abs{v_0}^2
\le
\frac{1}{r}\int_{B_{3r}} \abs{v_0}^2
=
\frac{1}{r}\int_{B_{3r}} \frac{|x|^{\alpha}\abs{v_0}^2}{|x|^{\alpha}}
\le
\frac{C_1}{r^{1+\alpha}}\|v_0\|^2_{L^{2,\alpha}}.
\end{align*}
Hence we have
\begin{equation}
\label{701}
 \frac{1}{r}\int_{B_r(x_0)} \abs{v_0}^2 \le \epsilon_*
\qquad \mbox{if} \  \ r \ge \max \left\{\frac{|x_0|}{2}, \left(\frac{C_1\|v_0\|_{L^{2,\alpha}}^2}{\epsilon_*}\right)^{\frac{1}{1+\alpha}} \right\}.
\end{equation}
Therefore, by virtue of \eqref{decay'}, we may use Theorem \ref{0311} (i) to see that
\begin{equation}
\label{04022}
v \ \mbox{is regular at} \  (x_0,t) \ \ \mbox{if} \ \  t \ge
\max \left\{\frac{|x_0|^2}{4}, \left(\frac{C_1\|v_0\|_{L^{2,\alpha}}^2}{\epsilon_*}\right)^{\frac{2}{1+\alpha}} \right\}.
\end{equation}
For $r \le |x_0|/2$, we have 
\begin{align*}
\frac{1}{r}\int_{B_r(x_0)} \abs{v_0}^2
\le
\frac{C_2}{r|x_0|^{\alpha}} \int_{B_r(x_0)}|x|^{\alpha} \abs{v_0}^2
\le
\frac{C_2}{r|x_0|^{\alpha}} \|v_0\|^2_{L^{2,\alpha}},
\end{align*}
and hence
\begin{align}
\label{702}
\frac{1}{r}\int_{B_r(x_0)} \abs{v_0}^2  \le \epsilon_* \qquad
\mbox{if} \ \  \frac{C_2\|v_0\|_{L^{2,\alpha}}^2}{|x_0|^{\alpha}\epsilon_*}  \le r \le \frac{|x_0|}{2}.
\end{align}
We may increase $C_1$ so that when  
$\frac{|x_0|}{2} \ge \left(\frac{C_1\|v_0\|_{L^{2,\alpha}}^2}{\epsilon_*}\right)^{\frac{1}{1+\alpha}}$, then $\frac{|x_0|}{2} \ge \frac{C_2\|v_0\|_{L^{2,\alpha}}^2}{|x_0|^{\alpha}\epsilon_*}$. For such $x_0$, \eqref{701} and \eqref{702} imply
\begin{align*}
\frac{1}{r}\int_{B_r(x_0)} \abs{v_0}^2  \le \epsilon_* \qquad
\mbox{if} \ \  r \ge \frac{C_2\|v_0\|_{L^{2,\alpha}}^2}{|x_0|^{\alpha}\epsilon_*}.
\end{align*}
Thus Theorem \ref{0311}\,(i) shows 
$v$ is regular for $t \ge 
\frac{C_2^2 \|v_0\|_{L^{2,\alpha}}^4}{|x_0|^{2\alpha}\epsilon_*^2}$.
This and \eqref{04022} show Part (ii).
 \end{proof}
 
\subsection{Energy concentration}
We now prove Theorem \ref{concentration} on the energy concentration at a blow up time by using Theorem \ref{large data}. 
\begin{proof}[Proof of Theorem \ref{concentration}]
(i)~Let $u$ be a local energy solution, and $\epsilon_*$ the small constant of Theorem \ref{large data}\,(ii). 
If $N_R^*:=\sup_{x \in \R^3,\,R\le r\le 1} \int_{B_1(x)} |u_0|^2 \le \epsilon_*$, then $u$ is regular in
$\R^3 \times [R^2,T)$ with $T:= \frac{c_1}{1+\epsilon_*^{12}}$.
Let $S_0=\frac{T}{2} \le \frac14$.

Suppose that \eqref{04082} is false. Then there exists
$t_1 <T_*$ such that
$$
\sup_{x_0\in \R^3}
\frac{1}{\sqrt{T_* -t_1}}
\int_{B_{\sqrt{\frac{T_*-t_1}{S}}}(x_0)}\abs{v(x,t_1)}^2dx \le \epsilon_*.
$$
Let  $\lambda=\sqrt{\frac{T_*-t_1}{S_0}}$
and let $u(x,t):=\lambda v(\lambda x,\lambda^2 t+t_1)$ and $u_0(x)=u(x,0)$.
Then $u$ satisfies
$$
N_{\sqrt{S_0}}^*=\sup_{x \in \R^3,\,\sqrt{S_0}\le r \le 1}\frac{1}{r}\int_{B_r(x)} |u_0|^2
\le
\sup_{x \in \R^3} \frac{1}{\sqrt{S_0}}\int_{B_1(x)} |u_0|^2 \le \epsilon_*,
$$
which implies that $u$ is regular in $\R^3 \times [\frac{T}{2},T)$.
Rescaling back, $v$ is regular in $\R^3 \times [T_*,2T_*-t_1)$.
This shows $v$ can be smoothly extended beyond
$t=T_*$ and shows the desired contradiction.

\medskip
(ii)~Although the proof mostly follow the outlines of those of (i) and
\cite[Theorem 2]{BP}, we will give the details, since some modification  is needed.

Let $T_2=T_2(M_*^{\frac12})<1$ be the constant in Theorem \ref{large data}\,(ii).
Let $S=T_2/2$ and $\de_*=\min(Sr_0^2,T_*)$.
 Assume the contrary so that
there exists a local energy solution $v$ which is singular at the point $(0,T_*)$, satisfies the type I condition \eqref{0318}, and
\begin{align}
\label{0319}
\frac{1}{\sqrt{T_* -t_0}}\int_{B_{\sqrt{\frac{T_*-t_0}{S}}}}\abs{v(x,t_0)}^2
\le \epsilon_*
\end{align}
 for some $t_0 \in (T_*-\de_*,T_*)$.
Let
$\lambda= \sqrt{\frac{T_*-t_0}{S}}$
and $u(x,t)=\lambda v(\lambda x, \lambda^2t +t_0)$.
We see
$\frac{1}{\sqrt{S}}\int_{B_1}\abs{u(x,0)}^2dx \le \epsilon_*$
since
$$
\int_{B_1} \abs{u(x,0)}^2dx =
\lambda^{-1} \int_{B_\lambda} \abs{v(y,t_0)}^2dy
=
 \sqrt{\frac{S}{T_*-t_0}} \int_{B_{\sqrt{\frac{T_*-t_0}{S}}}}
  \abs{v(y,t_0)}^2dy.
$$
Therefore 
$$
\sup_{\sqrt{\frac{T_2}{2}} \le r\le 1} \frac{1}{r}
\int_{B_r} |u(x,0)|^2 \le \epsilon_*.
$$
On the other hand,  we observe that
$$
\sup_{x_0\in \R^3} \int_{B_1(x_0)} |u(x,0)|^2=
\sup_{y_0\in \R^3} \lambda^{-1}\int_{B_{\lambda}(y_0)}|v(y,t_0)|^2dy
\le M_*.
$$
In the last inequality, \eqref{0318} can be used since $r=\lambda \le r_0$ and $T_*-r^2 < t_0 < T_*$.

Hence we may apply Theorem \ref{large data}\,(ii) to see that
$u$ is regular at $(0,t)$ for $t \in [T_2/2,T_2)$. Hence $v$ is regular at $(0,t)$ for $t \in [T_*, 2T_*-t_0)$.
This contradicts the assumption that $v$ is singular at $(0,T_*)$.
\end{proof}

\subsection{Regularity of discretely self-similar solutions}
\label{DSS}

Based on Theorem \ref{large data} and the proof of \cite[Corollary 1.5]{KMT20}, 
we now prove Theorem \ref{thm:DSS}. 
\begin{proof}[Proof of Theorem \ref{thm:DSS}](i)~%
For any $r>0$, let $k$ be the unique integer such that
$\lambda^{k-1} \le r < \lambda^k$, then by virtue of the discretely  
self-similarity,
\begin{align}
\label{0703}
\frac{1}{r} \int_{B_r(x_0)}|v_0|^2
&\le
\frac{1}{\lambda^{k-1}} \int_{B_{\lambda^k}(x_0)}|v_0|^2
=
\lambda \int_{B_{1}(\frac{x_0}{\lambda^k})}
|v_0(y)|^2dy
\\ \notag
&\le
\lambda \|v_0\|_{L^2_{\uloc}}^2.
\end{align}
Since $x_0$ and $r$ are arbitrary, this proves \eqref{04202}
with $M:=\norm{v_0}_{M^{2,1}} \le \sqrt \la \norm{v_0}_{L^2_\uloc}$.

Since $v_0 \in E^2$, there exists $R_*=R_*(v_0,\lambda)>0$ 
such that 
\begin{equation}
\|v_0\|_{L^2(B_1(x))}^2 \le \frac{\epsilon_*}{\lambda}
\qquad \mbox{if} \ |x| \ge R_*.
\label{0420}
\end{equation}
If $0<r\le \frac{|x_0|}{\lambda R_*}$ and $\lambda^{k-1} \le r < \lambda^k$, then
\[
\frac{|x_0|}{\lambda^{k}} \ge \frac{R_* r}{\la^{k-1}} \ge R_*.
\]
By \eqref{0703} and \eqref{0420},
we obtain
\begin{equation*}
\sup_{r\in (0,\frac{|x_0|}{\lambda R_*}]}
\frac{1}{r}\int_{B_r(x_0)}|v_0|^2
\le
\epsilon_*
\end{equation*}
Since $x_0$ is arbitrary in $\R^3 \backslash \{0\}$
and $R_*$ is independent of $x_0$,
this proves \eqref{0419} with $\mu=\frac{1}{\lambda R_*}$.

Applying Theorem \ref{large data} (ii) to
$u_0(x):=\mu|x_0|v_0(x_0+\mu|x_0|x)$ for each $x_0$, and 
noting that $\|u_0\|_{L^2_{\uloc}}
\le \|v_0\|_{M^{2,1}}$,
we see $u$ is regular at $x=0$ for $t \in (0, T_2(M)]$, that is,
 $v$ is regular at $(x_0,t)$ for $t \in (0, T_2(M)\mu^2|x_0|^2]$.
This concludes the proof of  (i).

\medskip
(ii) As the result of (i), it is direct that
\[
\abs{v(x,t)}\le \frac{C
}{\sqrt{t}}, \qquad 0<t\le C\abs{x}^2.
\]
As in the proof of Corollary 1.5 in \cite{KMT20}, it suffices to show that $v$ is regular in the set
\[
t\ge C\abs{x}^2,\qquad 1\le t\le \lambda^2.
\]
Localizing the solution in the region above, it turns out that the localized solution $\td v$ is a weak solution of the Navier-Stokes equations with a regular source term. Since the weak solutions satisfy the 
classical energy inequality, 
there is time $t_*\in [1, \lambda^2)$  such that the solution is in $H^1$ at the time.
Then due to the local existence of strong solution and the weak-strong uniqueness, there is a constant $\delta>0$ independent of $\mu$
such that $\td v$ is regular in $[t_*, t_*+\delta]$. Therefore, if $\lambda_*$ is taken to be sufficiently small close to $1$, the solution is regular globally in time. 
 For any $t$, there is an integer $k$ such that 
$\lambda^{2k} t \in [1, \lambda^2]$, and we then have, due to scaling invariance, 
\[
\abs{v(x, t)}=\abs{\lambda^k v(\lambda^kx, \lambda^{2k}t)}\le C\lambda^k\le C\frac{\lambda_*}{\sqrt{t}}\le \frac{2C}{\sqrt{t}}.
\qedhere
\]
\end{proof}

\begin{remark}
\label{example}
If  $v_0 \in L^2_{\uloc} \backslash E^2$, 
\eqref{0419} does not hold in general.
To see this we show that there exists a DSS function $F \in L^2_{\uloc} \backslash E^2$ such that 
\begin{equation*}
\sup_{x \in \R^3, r\in(0,\mu|x|] }
\frac{1}{r}\int_{B_r(x)}|F|^2
>
\epsilon_* \qquad \mbox{for \ any} \ \mu \in (0,1).
\end{equation*}
Indeed for $K>\sqrt{ \epsilon_*/4\pi}$, define 
\begin{equation}
\label{05102}
F(x):=\sum_{k \in \Z} \lambda^k f(\lambda^k x)
\qquad 
\mbox{with}
\ 
f(x)=K|x-\lambda^{\frac12}e_1|^{-1}\chi(x)
\end{equation}
where $\chi$ is the characteristic function supported in 
$B_{\lambda} \backslash B_1$. 
This function is given in
\cite[Comments on Theorem 1.2]{BT17} and \cite[Lemma 6.3]{BT19} as an example 
in $M^{2,1} \backslash E^2$ while $L^{3,\infty} \subset E^2 $.
It is also shown in \cite[Lemma 6.1]{BT19} that $\textup{DSS} \cap L^2_\uloc=\textup{DSS} \cap M^{2,1}$.
For any $\mu \in (0, 1-\la^{-1/2})$ we have 
\begin{align*}
\sup_{x \in \R^3,\,r\in(0,\mu|x|] }
\frac{1}{r}\int_{B_r(x)}|F|^2
&\ge 
\frac{1}{\mu\lambda^{\frac12}}
\int_{B_{\mu\lambda^{\frac12}}(\lambda^{\frac12}e_1)} |F|^2
%
%
%
\\
&=4\pi K^2>\epsilon_*.
\end{align*}
Hence \eqref{0419} does not hold for $F$ for any $\mu>0$. 

\end{remark}

\section{Appendix:~alternative proof of Theorem \ref{large data}}
By the definition of the local energy solution, there exists
$c_r=c_r(t)$
such that the pressure admits the following decomposition \eqref{decomposition}:
\begin{align}
p-c_r+\frac{|v|^2}{3}
&=p_{\loc}+p_{nonloc}
\notag
\\
&:=\mbox{p.v.}\int_{B_{3r}} K(x-y) (v\otimes v)(y)dy
+\int_{\R^3 \backslash B_{3r}} (K(x-y)-K(-y) )(v\otimes v)(y)dy.
\label{pressure}
\end{align}
Since $(v, p-c_r)$ is a suitable weak solution to \eqref{NS} in 
$B_{2r}$, 
the local energy inequality with the test function
given in the proof of Theorem \ref{local data}
readily yields 
\begin{align}
E_r(t)
&\le
\frac{2}{2r}\int_{B_{2r}} |v_0|^2 
+
\frac{C}{r^3}\int^t_0\! \int_{B_{2r}}|v|^2
+
\frac{C}{r^2}\int^t_0\!\int_{B_{2r}}\!\!|v|^3+|p_{\loc}|^{\frac32}+
|p_{nonloc}|^{\frac32}
\notag
\\
&=:2N_{2r}+I_{lin} +I_{nonlin}+ I_{ploc}+I_{pnonloc}.
\label{A_r'}
\end{align}
We divide the estimate into two cases.

\medskip
\noindent
{\underline{\bf Case I:~$R\le r\le \frac16$}.}
For the simplicity of notation
Let $\mathcal{E}_r(t):= \sup_{r\le \rho \le 1/2} M_{\rho}(t)$.
In the same way as \eqref{est:0324} and \eqref{est:03043} we have
\begin{align}
I_{lin}
&\le
\frac{C}{R^2}\int^t_0
\mathcal{E}_R(s)ds,
\label{vlin}
\end{align}
\begin{align}
I_{nonlin}
&\le
\epsilon \mathcal{E}_R(t)
+
\frac{C}{r^2} \int^t_0 \mathcal{E}_{R}(s)^3
+ \mathcal{E}_{R}(s)^\frac32 ds.
\label{vnonlin}
\end{align}

For the local pressure term,
the Calder\'on-Zygmund estimate and \eqref{vnonlin} give
\begin{align}
I_{ploc}
&\le
\frac{C}{r^2}\int^t_0 \int_{B_{3r}}\!\!|v|^3
\notag
\\
&\le
\epsilon \mathcal{E}_R(t)
+
\frac{C_{\epsilon}}{R^2} \int^t_0 \mathcal{E}_{R}(s)^3
+ \mathcal{E}_{R}(s)^\frac32 ds.
\label{ploc'}
\end{align}
On the other hand, since $|x-y| \ge |y|/3$
for $x \in B_{2r}$ and $y \in \R^3 \backslash B_{3r}$, we see
\begin{align}
|p_{nonloc}(x)|
&\le
\int_{\R^3 \backslash B_{3r}}
|K(x-y)-K(-y)||v(y)|^2dy
\notag
\\
&\le
Cr\int_{\R^3 \backslash B_{3r}}
\frac{1}{|x-y|^4} |v(y)|^2dy
\notag
\\
&\le
Cr\int_{\R^3 \backslash B_{3r}}
\frac{|v(y)|^2}{|y|^4} dy
\notag
\\
&\le
Cr\sum_{k=2}^{\lfloor -\log_2 r-1 \rfloor} \int_{B_{2^{k}r} \backslash B_{2^{k-1}r}}
\frac{|v(y)|^2}{|y|^4} dy
+
Cr\int_{\R^3 \backslash B_{1/4}}
\frac{|v(y)|^2}{|y|^4}dy
\notag
\\
&\le
\frac{C}{r^2}\sum_{k=2}^{\lfloor -\log_2 r-1 \rfloor}\frac{1}{2^{3k}}
\Big(\frac{1}{2^{k}r} \int_{B_{2^k r}} |v(y)|^2 dy\Big)
+
CE_{\uloc}
\notag
\\
&\le
\frac{C}{r^2}
\mathcal{E}_R
+
C\|v_0\|_{L^2_{\uloc}}^2,
\label{pnonloc}
\end{align}
provided $t\le T_{\uloc}$, where we used Lemma \ref{apriori uloc} in the last line.
We then obtain
\begin{align}
\frac{C}{r^2}\int^t_0 \int_{B_{2r}} |p_{nonloc}|^{\frac32}
&\le
\frac{C}{r^2}\int^t_0 \mathcal{E}_R(s)^{\frac32}ds +
t \|v_0\|^3_{L^2_{\uloc}}
\notag
\\
&\le
\frac{C}{r^2}\int^t_0 \mathcal{E}_R(s)^{\frac32} ds+
\frac{\delta}{10},
\label{pnonloc'}
\end{align}
provided $t \le \min\{ \frac{c\delta}{\|v_0\|_{L^2_{\uloc}}^3},T_{\uloc} \}$
with a small absolute constant $c>0$.
Hence applying \eqref{vlin}, \eqref{vnonlin}, \eqref{ploc'}, and \eqref{pnonloc'}
to \eqref{A_r'},  we obtain
\begin{align*}
\sup_{R\le r <1/6}E_r(t)
&\le
\frac{\delta}{2}
+\frac{C}{R^2} \int^t_0 \mathcal{E}_R(s)
+
\mathcal{E}_R(s)^3 ds
\end{align*}
for $\delta \ge 5N_R$.

\medskip
\noindent
{\underline{\bf Case II:~$\frac16\le r\le \frac12$}.}
In order to bound the right hand side of \eqref{A_r'},
we observe from Lemma \ref{apriori uloc} that
\begin{equation}
\label{est:uloc}
\sup_{1/3\le r\le 2} E_r(t) \le C E_{\uloc}(t) \le C\|v_0\|_{L^2_{\uloc}}^2
\qquad \mbox{holds} \
\mbox{for} \ t\le T_{\uloc}.
\end{equation}
This shows 
\begin{align*}
I_{lin}
&\le
C \int^t_0\sup_{1/3\le r\le1}
M_{r}(s) ds
\\
&\le
Ct \|v_0\|_{L^2_{\uloc}}^2,
\end{align*}
and hence if $t\le \min\{T_{\uloc}, \frac{c\delta}{\|v_0\|_{L^2_{\uloc}}^2}\}$
with a suitable small constant $c>0$, we
have
\begin{align}
\label{vlin2}
I_{lin}&\le
\frac{\delta}{10}.
\end{align}
For the nonlinear term in \eqref{A_r'}, arguing as \eqref{est:03044} we have
\begin{align}
I_{nonlin}
& \le
C \bke{\int_0^t\int_{B_{2}} \abs{ \nabla v}^2}^{\frac34}
\bke{\int^t_0 \bke{\int_{B_{2}} \abs{v}^2}^3}^{\frac{1}{4}}
 +C\int^t_0\bke{\int_{B_{2}} \abs{v}^2}^{\frac{3}{2}}ds
\notag
\\
&\le
C(t^{\frac14} +t) \|v_0\|_{L^2_{\uloc}}^3,\label{est:0310}
\end{align}
from which the Calder\'on-Zygmund estimate also gives
\begin{align}
\label{est:03102}
I_{ploc}&\le
\frac{C}{r^2}\int^t_0\!\! \int_{B_{3r}}\!\!|v|^3
\le
C(t^{\frac14} +t) \|v_0\|_{L^2_{\uloc}}^3.
\end{align}
Thus the right hand sides in \eqref{est:0310} and \eqref{est:03102} are
 bounded by $\frac{\delta}{10}$ provided
$$
t\le \min \left \{T_{\uloc}, \frac{c\delta^4}{\|v_0\|_{L^2_{\uloc}}^{12}},
\frac{c\delta}{\|v_0\|_{L^2_{\uloc}}^{3}} \right\}
$$
with some small absolute constant $c>0$.
On the other hand, in the same way as in \eqref{pnonloc},
we have
\begin{align*}
|p_{nonloc}(x)|
\le
C\int_{\R^3 \backslash B_{\frac12}}
\frac{|v(y)|^2}{|y|^4} dy
&\le
C\|v_0\|_{L^2_{\uloc}}^2,
\end{align*}
which implies
\begin{align*}
I_{pnonloc}
&\le
\frac{\delta}{10} \qquad \mbox{for} \  \ t \le \frac{c\delta}{\|v_0\|_{L^2_{\uloc}}^3}.
\end{align*}
Making use of these estimates in \eqref{A_r'}, we obtain
$$
\sup_{1/6\le r \le 1/2}E_r(t)
\le
\frac{\delta}{2}.
$$

Note here that by choosing $c_1>0$ sufficiently small, we may take
$$
T_1:= \frac{c_1\min\{1, \delta^4 \}}{1+\|v_0\|_{L^2_{\uloc}}^{12}}
\le \min \left\{T_{\uloc}, \frac{c\delta}{\|v_0\|_{L^2_{\uloc}}^2},\frac{c\delta^4}{\|v_0\|_{L^2_{\uloc}}^{12}},
\frac{c\delta}{\|v_0\|_{L^2_{\uloc}}^{3}} \right\}.$$
Therefore combining the conclusions of the cases I and II, we have
\begin{align*}
\mathcal{E}_R(t)
&\le
\frac{\delta}{2}
+\frac{C}{R^2} \int^t_0 \mathcal{E}_R(s)
+
\mathcal{E}_R(s)^3 ds
\end{align*}
 for $t\le T_1$.
Applying Lemma \ref{gronwall}
we obtain
\begin{align}
\label{A_R}
\mathcal{E}_R(t)
&\le \delta \quad  \mbox{for} \ t \le \min \{ \lambda R^2, T_1 \}.
\end{align}
Since $\delta \ge 5N_R \ge 5N_r$
for any $r\in [R,1/2]$, we may replace $R$ by $r$ in \eqref{A_R}, and thus we have
verified \eqref{A_r} for $R \le r \le \frac12$.

Since the remaining proof is the same as Theorem \ref{local data},
we omit the details.
\hfill $\Box$

\renewcommand{\baselinestretch}{0.9}

\section*{Acknowledgments}
We warmly thank Zachary Bradshaw and Christophe Prange for very helpful suggestions.
The research of Kang was partially supported by NRF-2019R1A2C1084685. The research of Miura was partially supported
by JSPS grant 17K05312. The research of Tsai was partially supported
by NSERC grant RGPIN-2018-04137.

\end{document}